\pdfoutput=1 

\documentclass[11pt]{amsart}

\usepackage{amssymb,amsthm,enumitem,colonequals,tikz-cd,microtype,comment}
\usepackage[normalem]{ulem} 

\usepackage[osf]{Baskervaldx}
\usepackage[baskervaldx]{newtxmath}
\usepackage[cal=boondoxo]{mathalfa} 

\usepackage[top=3.75cm, bottom=3cm, left=3.5cm, right=3.5cm]{geometry}

\usepackage{xcolor}
\colorlet{darkblue}{blue!55!black}
\colorlet{darkcyan}{cyan!50!black}
\colorlet{darkgreen}{green!60!black} 

\PassOptionsToPackage{hyphens}{url}
\usepackage{hyperref}
\hypersetup{
    colorlinks=true,
    linkcolor=darkblue,
    urlcolor=darkcyan,
    citecolor=darkgreen,
}

\def\eqref#1{\textcolor{darkblue}{(\ref{#1})}}

\usepackage[nameinlink]{cleveref} 
\Crefformat{section}{#2\S#1#3}
\Crefmultiformat{section}{#2\S\S#1#3}{ and~#2#1#3}{, #2#1#3}{, and~#2#1#3}

\usepackage[pagewise]{lineno}
\let\oldequation\equation
\let\oldendequation\endequation
\renewenvironment{equation}{\linenomathNonumbers\oldequation}{\oldendequation\endlinenomath}
\expandafter\let\expandafter\oldequationstar\csname equation*\endcsname
\expandafter\let\expandafter\oldendequationstar\csname endequation*\endcsname
\renewenvironment{equation*}{\linenomathNonumbers\oldequationstar}{\oldendequationstar\endlinenomath}
\let\oldalign\align
\let\oldendalign\endalign
\renewenvironment{align}{\linenomathNonumbers\oldalign}{\oldendalign\endlinenomath}
\expandafter\let\expandafter\oldalignstar\csname align*\endcsname
\expandafter\let\expandafter\oldendalignstar\csname endalign*\endcsname
\renewenvironment{align*}{\linenomathNonumbers\oldalignstar}{\oldendalignstar\endlinenomath}

\theoremstyle{plain}
\newtheorem{theorem}{Theorem}[section]
\newtheorem{lemma}[theorem]{Lemma}
\newtheorem{corollary}[theorem]{Corollary}
\newtheorem{proposition}[theorem]{Proposition}

\theoremstyle{definition}
\newtheorem{definition}[theorem]{Definition}
\newtheorem{example}[theorem]{Example}
\newtheorem{remark}[theorem]{Remark}
\newtheorem{hypothesis}[theorem]{Hypothesis}
\newtheorem{notation}[theorem]{Notation}
\newtheorem{reminder}[theorem]{Reminder}

\newtheorem*{ack}{Acknowledgments}

\AddToHook{env/conjecture/begin}{\crefalias{theorem}{conjecture}}
\AddToHook{env/lemma/begin}{\crefalias{theorem}{lemma}}
\AddToHook{env/corollary/begin}{\crefalias{theorem}{corollary}}
\AddToHook{env/proposition/begin}{\crefalias{theorem}{proposition}}
\AddToHook{env/definition/begin}{\crefalias{theorem}{definition}}
\AddToHook{env/remark/begin}{\crefalias{theorem}{remark}}
\AddToHook{env/example/begin}{\crefalias{theorem}{example}}
\AddToHook{env/hypothesis/begin}{\crefalias{theorem}{hypothesis}}
\AddToHook{env/notation/begin}{\crefalias{theorem}{notation}}
\AddToHook{env/reminder/begin}{\crefalias{theorem}{reminder}}

\numberwithin{equation}{section}
\numberwithin{theorem}{section}

\title[Frobenius generation for algebraic stacks]{Frobenius generation for algebraic stacks}

\author[P.~Lank]{Pat Lank}
\address{P.~Lank,
Dipartimento di Matematica “F. Enriques”, Universit\`{a} degli Studi di Milano, Via Cesare
Saldini 50, 20133 Milano, Italy}
\email{plankmathematics@gmail.com}

\author[F. ~Peng]{Fei Peng}
\address{F.~Peng,
School of Mathematics \& Statistics\\
The University of Melbourne\\
Parkville, VIC, 3010\\
Australia}
\email{pengf2@student.unimelb.edu.au}

\date{\today}



\begin{document}
    
\begin{abstract}
    We introduce a notion of $F$-finiteness for algebraic stacks in positive characteristic. 
    Our main result shows that sufficiently many Frobenius pushforwards generate the bounded derived categories of coherent sheaves on Noetherian concentrated $F$-finite algebraic stacks with quasi-finite and separated diagonal. 
    This generalizes, and independently recovers, a result of Ballard--Iyengar--Lank--Mukhopadhyay--Pollitz.
\end{abstract}

\maketitle

\tableofcontents

\section{Introduction}
\label{sec:intro}

This work is concerned with explicit generators of the bounded derived category $D^b_{\operatorname{coh}}(\mathcal{X})$ of a Noetherian algebraic $\mathbb{F}_p$-stack $\mathcal{X}$. 
While the existence of classical and strong generators is now well understood for algebraic stacks, comparatively few constructions explicitly identify generators (see e.g.\ \cite{Hanlon/Hicks/Lazarev:2024}).
Our work addresses this gap.

\subsection{What is known}
\label{sec:intro_known}

The notion of generation in a triangulated category $\mathcal{T}$ was introduced in \cite{Bondal/VandenBergh:2003}. 
Roughly speaking, an object $G \in \mathcal{T}$ is called a \textit{classical generator} if every object of $\mathcal{T}$ can be obtained from $G$ using only finite direct sums, direct summands, and cones. 
Furthermore, if there exists an integer $n \geq 0$ such that this process can be completed using at most $n+1$ cones, then $G$ is called a \textit{strong generator}.

Generation techniques allow one to study objects in a triangulated category by means of a single object. 
The importance of these techniques has been demonstrated in various contexts such as commutative algebra, representation theory, and more recently geometry. This includes: representability theorems for cohomological functors when strong generators exist \cite{Rouquier:2008}, applications to singularity theory \cite{Lank/Venkatesh:2025,Lank/McDonald/Venkatesh:2025,DeDeyn/Lank/Lank/ManaliRahul/Venkatesh:2026}, and the resolution of open conjectures in algebraic geometry \cite{Neeman:2021a,Neeman:2022}. 

In the setting of schemes, the existence of generators for $D^b_{\operatorname{coh}}$ has become more well-understood, see e.g.\ \cite{Bondal/VandenBergh:2003,Rouquier:2008,Keller/Murfet/VanDenBergh:2009,Lunts:2010,Elagin/Lunts/Schnurer:2020,Iyengar/Takahashi:2016,Aoki:2021}. 
Recently, the existence of generators for $D^b_{\operatorname{coh}}$ has been established generally for algebraic stacks \cite{Hall/Priver:2024,DeDeyn/Lank/ManaliRahul:2025,DeDeyn/Lank/ManaliRahul/Peng:2025}. 
However, it remains elusive to explicitly identify generators. 

In positive characteristic, \cite{Ballard/Iyengar/Lank/Mukhopadhyay/Pollitz:2023} recently provided a construction of explicit generators. Recall that a Noetherian scheme of positive characteristic is \textit{$F$-finite} if its Frobenius morphism is finite. For example, every Noetherian scheme locally of finite type over a perfect field of positive characteristic is $F$-finite. 

It was shown in \cite{Ballard/Iyengar/Lank/Mukhopadhyay/Pollitz:2023} that on an $F$-finite scheme, the Frobenius pushforward of a compact generator becomes a classical generator for $D^b_{\operatorname{coh}}$ after sufficiently many iterations. 
This provides a categorical explanation of the fact that the Frobenius morphism on schemes detects singularities. 
It also recovers a regularity criterion of Kunz \cite{Kunz:1969}.

\subsection{What we do}
\label{sec:intro_what_we_do}

In this article, we extend the main result of \cite{Ballard/Iyengar/Lank/Mukhopadhyay/Pollitz:2023} to the case of algebraic stacks. 
In fact, our methods for the main result (see \Cref{introthm:Frobenius_generation}) are independent of loc.\ cit.
This is explained more carefully later.

Currently, $F$-finiteness for algebraic stacks has not been defined, nor studied. 
To this end, we propose the following analog: We say an algebraic stack $\mathcal{X}$ is \textit{$F$-finite} if the Frobenius morphism on $\mathcal{X}$ is concentrated and locally of finite type.
By \Cref{ex:DM_F-finite}, $F$-finiteness occurs for all Deligne--Mumford stacks or concentrated algebraic stacks locally of finite type over a perfect field of positive characteristic.

Along the way towards our main results, we provide a characterization of $F$-finiteness for classifying stacks of group schemes. 
We believe this would be of independent interest. 
See \Cref{sec:finiteness_F-finiteness} for more details. 

\subsubsection{Frobenius generation}
\label{sec:intro_what_we_do_frob_gen}

Now we state our main result:

\begin{theorem}
    \label{introthm:Frobenius_generation}
    Let $\mathcal{X}$ be a Noetherian $F$-finite algebraic stack with separated quasi-finite diagonal. 
    If $\mathcal{X}$ is concentrated, then for any $Z\subseteq |\mathcal{X}|$ closed and $e \gg 0$, there is a $G\in \operatorname{Perf}_Z (\mathcal{X})$ such that $\mathbf{R} F_\ast^e G$ is classical generator for $D^b_{\operatorname{coh},Z}(\mathcal{X})$.
\end{theorem}

See \Cref{thm:Frobenius_generation}. Our argument for \Cref{introthm:Frobenius_generation} differs from that of \cite[Theorem A]{Ballard/Iyengar/Lank/Mukhopadhyay/Pollitz:2023}.
In particular, the results of \cite{Ballard/Iyengar/Lank/Mukhopadhyay/Pollitz:2023} rely on Stevenson's local-to-global principle for tensor actions of triangulated categories \cite{Stevenson:2014a}. 

By contrast, we leverage Hall--Rydh's \'{e}tale d\'{e}vissage to construct generators geometrically \cite{Hall/Rydh:2018}.
This naturally leads to the relative statement. 

Our argument proceeds in three stages:
\begin{itemize}
    \item (Local rings) The key ingredient is \cite[Theorem 4.1.13]{Bhatt/Blickle/Schwede/Tucker:2026}. It allows us to check that the residue field lies in the thick subcategory generated by the sufficiently many Frobenius pushforwards of Koszul complexes. 
    \item (Affine schemes) The result for $F$-finite local rings is extended to $F$-finite affine schemes using a commutative algebra result of Letz \cite{Letz:2021}.
    \item (Stacks) Generators are glued in the Zariski topology in \Cref{lem:generation_is_affine_local}, and in the quasi-finite and flat topology using \'{e}tale d\'{e}vissage in \Cref{prop:open_immersion,prop:finite_cover,prop:etale_nbhd}. 
\end{itemize}
These techniques should be useful to establish local-global generation type results for algebraic stacks beyond the scope of our results. 

\begin{remark}
    \hfill
    \begin{enumerate}
        \item The $F$-finiteness assumption in \Cref{introthm:Frobenius_generation} is necessary even in the affine case (e.g.\ for $\mathbb{F}_p(x_1,x_2,\ldots)$). 
        We also provide a counterexample of an algebraic stack locally of finite type over $\mathbb{F}_p$ but not $F$-finite. See \Cref{ex:non_F-finite_false} for more details. 
        \item The concentrated condition ensures that the compact objects coincide with the perfect complexes \cite[Remark 4.6]{Hall/Rydh:2017}. 
        This assumption is necessary as a limitation of our method. 
        See \Cref{rm:finite_cover_quotients_for_support} for more details. 
        \item We do not know if \Cref{introthm:Frobenius_generation} would hold for $F$-finite but non-concentrated algebraic stacks. See \Cref{ex:F-finite_non_tame_maybe}. 
        One technical obstruction is whether we could descend the condition described in \Cref{hyp:frob_gen} along finite \'{e}tale coverings. 
        It would be interesting to investigate these questions further. 
    \end{enumerate}
\end{remark}

As a consequence of \Cref{thm:Frobenius_generation}:

\begin{corollary}
    \label{cor:kunz}
    Let $\mathcal{X}$ be a Noetherian concentrated $F$-finite algebraic stack with separated quasi-finite diagonal. 
    Then the following are equivalent for any $Z\subseteq |\mathcal{X}|$ closed:
    \begin{enumerate}
        \item $Z$ is contained in the regular locus of $\mathcal{X}$ (i.e.\ $x\in \mathcal{X}$ for which $\mathcal{X}$ is regular at $x$)
        \item $\mathbf{R} F_\ast^e G\in \operatorname{Perf}_Z (\mathcal{X})$ for all $e \gg 0$ and $G$ a classical generator for $\operatorname{Perf}_Z (\mathcal{X})$.
    \end{enumerate}
\end{corollary}

This yields a categorical criterion for regularity for algebraic stacks in positive characteristics. 
It is inspired by the aforementioned classical result of Kunz on regularity \cite{Kunz:1976}.

\subsubsection{Bounding iterates}
\label{sec:intro_what_we_do_iterate}

A natural question is how many Frobenius pushforwards are required before generation occurs. 
In \cite{Ballard/Iyengar/Lank/Mukhopadhyay/Pollitz:2023}, a numeric called the \textit{codepth} of a Noetherian scheme was introduced to bound this number of iterates. 
It is not clear whether codepth behaves well for algebraic stacks.

Instead, we work with a slightly less sharp but more tractable invariant $\gamma(F_\ast \mathcal{O}_X)$.
See \Cref{not:gamme_function}. 
Roughly speaking, it defined using differences of minimal number of generators of Frobenius pushforwards of a module and the depth of a local ring. 
The upshot is that \'{e}tale d\'{e}vissage gains us control of this invariant in the \'{e}tale topology. 

This brings us to the next result:

\begin{proposition}
    \label{introprop:Frobenius_generation_bound}
    Let $\mathcal{X}$ be a concentrated $F$-finite Deligne--Mumford stack with separated diagonal. 
    Consider an \'{e}tale surjective morphism $U\to \mathcal{X}$ from an affine scheme where $F_\ast \mathcal{O}_U$ is minimally generated by $N$ local sections. 
    Then, in \Cref{introthm:Frobenius_generation}, one can take $e> \log_p (N) $. 
    In particular, $N$ is always finite and independent of $Z$.
\end{proposition}

See \Cref{prop:Frobenius_generation_bound}. 
The case in which $F_\ast^e \mathcal{O}_X$ is a strong generator of $D^b_{\operatorname{coh}}(X)$ for $e\gg 0$ on an $F$-finite scheme was studied in \cite{Ballard/Iyengar/Lank/Mukhopadhyay/Pollitz:2023}. 
This holds when $X$ is quasi-affine, but not necessarily in the global setting. In particular, a smooth projective curve over an algebraically closed field satisfies this condition if, and only if, the curve has genus zero. 

It is natural to ask when analogous conditions hold for algebraic stacks in prime characteristic. 
That is, $F_\ast^e \mathcal{O}_{\mathcal{X}}$ is a classical or strong generator for $D^b_{\operatorname{coh}}(\mathcal{X})$ with $\mathcal{X}$ an $F$-finite algebraic stack. 
We partially address this condition for tame stacky curves over an algebraically closed field of prime characteristic in \Cref{prop:F-thickness_stacky_curve}. 

In a related direction, \cite{Hanlon/Hicks/Lazarev:2024} and \cite{Favero/Huang:2023} show that this property holds for many toric stacks, even beyond the Deligne--Mumford setting. 
It would be interesting to determine the precise class of toric stacks for which this property is satisfied.

\begin{ack}
    The authors thank Andres Fernandez Herrero, Timothy De Deyn, Jack Hall, Andrew Hanlon, Urs Hartl, Kabeer Manali Rahul, Alapan Mukhopadhyay, Amnon Neeman, Frank Neumann, Martin Olsson, David Rydh. 
    Lank was supported under the ERC Advanced Grant 101095900-TriCatApp.
    Peng was supported by the Australian Research Council DP210103397 and FT210100405, a Melbourne Research Scholarship, and the Science Abroad Travelling Scholarship offered by the University of Melbourne.
\end{ack}

\section{Preliminaries}
\label{sec:prelim}

Throughout our work, $p$ is a prime number unless otherwise specified.

\subsection{Generation}
\label{sec:prelim_generation}

We briefly discuss a notion of generation for triangulated categories. For details, the reader is referred to \cite{Bondal/VandenBergh:2003}. Consider a triangulated category $\mathcal{T}$. 
Denote its shift functor by $[1]\colon \mathcal{T} \to \mathcal{T}$. Let $\mathcal{S}\subseteq \mathcal{T}$. 
We say $\mathcal{S}$ is \textbf{thick} if it is a triangulated subcategory of $\mathcal{T}$ which is closed under direct summands. 
Denote by $\langle \mathcal{S} \rangle$ the smallest thick subcategory containing $\mathcal{S}$ in $\mathcal{T}$; if $\mathcal{S}$ consists of a single object $G$, then $\langle \mathcal{S}\rangle$ will be written as $\langle G\rangle$. 
Set $\operatorname{add}(\mathcal{S})$ to be the smallest strictly full subcategory of $\mathcal{T}$ containing $\mathcal{S}$ which is closed under shifts, finite coproducts, and direct summands. 
Define $\langle \mathcal{S} \rangle_0$ to consist of all objects in $\mathcal{T}$ isomorphic to the zero object. 
Set $\operatorname{smd}(\mathcal{S})$ to be the strictly full subcategory of direct summands of finite coproducts of objects from $\mathcal{S}$. 

Moreover, $\langle \mathcal{S} \rangle_1 := \operatorname{add}(\mathcal{S})$, and inductively, let 
\begin{displaymath}
    \langle \mathcal{S} \rangle_n := \operatorname{add} \{ \operatorname{cone}(\phi) \colon \phi \in \operatorname{Hom}_{\mathcal{T}} (\langle \mathcal{S} \rangle_{n-1}, \langle \mathcal{S} \rangle_1) \}.
\end{displaymath}
There is a filtration for the smallest thick subcategory; namely, $\langle \mathcal{S} \rangle = \cup^\infty_{n=0} \langle \mathcal{S} \rangle_n$. 
An object $G\in \mathcal{T}$ is called a \textbf{classical generator} if $\langle G \rangle = \mathcal{T}$. 
If there is an $n\geq0$ such that $\langle G \rangle_{n+1} = \mathcal{T}$, then $G$ is called a \textbf{strong generator}.

If $\mathcal{T}$ admits small coproducts, then the collection of compact objects in $\mathcal{T}$ is denoted by $\mathcal{T}^c$. 
These form a triangulated subcategory of $\mathcal{T}$. 
We say that $\mathcal{T}$ is \textbf{compactly generated} if it coincides with the smallest triangulated subcategory of $\mathcal{T}$ containing $\mathcal{T}^c$ and closed under small coproducts. 
Equivalently, $\mathcal{T}$ is compactly generated if, for any $E \in \mathcal{T}$ satisfying $\operatorname{Hom}(P, E) = 0$ for all $P \in \mathcal{T}^c$, one has $E \cong 0$ \cite[Lemma 2.2.1]{Schwede/Shipley:2003}. 
If $\mathcal{T}$ is compactly generated, then classical generators for $\mathcal{T}^c$ coincide with compact generators for $\mathcal{T}$ \cite[\href{https://stacks.math.columbia.edu/tag/09SR}{Tag 09SR}]{stacks-project}. 
Let $\operatorname{Add}(\mathcal{S})$ be the smallest strictly full subcategory of $\mathcal{T}$ containing $\mathcal{S}$ that is closed under shifts, small coproducts, and direct summands. 
Inductively, let $\overline{\langle \mathcal{S} \rangle}_0$ consist of all objects in $\mathcal{T}$ isomorphic to the zero object, $\overline{\langle \mathcal{S} \rangle}_1 := \operatorname{Add}(\mathcal{S})$, and
\begin{displaymath}
    \overline{\langle \mathcal{S} \rangle}_n := \operatorname{Add} \{ \operatorname{cone}(\phi) \mid \phi \in \operatorname{Hom}_{\mathcal{T}} (\overline{\langle \mathcal{S} \rangle}_{n-1}, \overline{\langle \mathcal{S} \rangle}_1) \}.
\end{displaymath}

\begin{example}
    Let $X$ be a quasi-compact quasi-separated scheme. 
    Then $D_{\operatorname{qc}}(X)^c=\operatorname{Perf}(X)$, and moreover, $\operatorname{Perf}(X)$ admits a classical generator. 
    See \cite[Theorem 3.1.1]{Bondal/VandenBergh:2003}. 
    If $X$ is quasi-affine, then $\mathcal{O}_X$ compactly generates $D_{\operatorname{qc}}(X)$ \cite[\href{https://stacks.math.columbia.edu/tag/0BQT}{Tag 0BQT}]{stacks-project}.
    More generally, let $Z\subseteq X$ be closed such that its complement $U=X\setminus Z$ is quasi-compact.
    By \cite[Theorem 6.8]{Rouquier:2008}, $D_{\operatorname{qc},Z}(X)$ is compactly generated by a single object. 
    Hence, $\operatorname{Perf}(X)\cap D_{\operatorname{qc},Z}(X)$ admits a classical generator.
\end{example}

\begin{remark}
    \label{rmk:relative_perf_dense_localization}
    Let $X$ be a Noetherian scheme. 
    Consider an open immersion $j\colon U \to X$. Define $Z:= |X|\setminus |U|$. 
    For any $W\subseteq |X|$ closed, there exists a Verdier localization sequence of compactly generated triangulated categories,
    \begin{displaymath}
        D_{\operatorname{qc}, Z\cap W}(X) \to D_{\operatorname{qc},W} (X) \xrightarrow{\mathbf{L}j^\ast} D_{\operatorname{qc},W\cap U}(X).
    \end{displaymath}
    See e.g.\ \cite[Proposition 3.1]{Lank:2026}. 
    By \cite[Theorem 2.1]{Neeman:1996}, there exists a Verdier localization sequence up to direct summands,
    \begin{displaymath}
        \operatorname{Perf}_{Z\cap W}(X) \to \operatorname{Perf}_W (X) \xrightarrow{\mathbf{L}j^\ast} \operatorname{Perf}_{W\cap U}(X).
    \end{displaymath}
    In fact, for any $G\in \operatorname{Perf}_W (X)$ a compact generator for $D_{\operatorname{qc},W}(X)$ and $p\in X$, $\mathbf{L}t^\ast G$ is a compact generator for $D_{\operatorname{qc},W\cap \operatorname{Spec}(\mathcal{O}_{X,p})}(\mathcal{O}_{X,p})$. 
    Indeed, this follows from \cite[Lemma 4.8(2)]{Hall/Rydh:2017} and \cite[Lemma 1.2]{Neeman:1992}.
\end{remark}

\subsection{Algebraic stacks}
\label{sec:prelim_stacks}

We follow \cite{stacks-project} for conventions regarding algebraic stacks and \cite[\S1]{Hall/Rydh:2017} for the derived pullback/pushforward adjunction. 
Typically, $X$, $Y$, etc.\ refer to schemes/algebraic spaces, whereas $\mathcal{X}$, $\mathcal{Y}$, etc.\ refer to algebraic stacks. 
Let $\mathcal{X}$ be a Noetherian algebraic stack. 

\subsubsection{Finiteness reminder}
\label{sec:prelim_stacks_finiteness}

A smooth morphism of algebraic stacks is locally of finite presentation, and hence, locally of finite type \cite[\href{https://stacks.math.columbia.edu/tag/0DNP}{Tag 0DNP} \& \href{https://stacks.math.columbia.edu/tag/06Q5}{Tag 06Q5}]{stacks-project}. 
Additionally, if $f$ is a morphism from a quasi-compact quasi-separated source to a quasi-separated target, then $f$ is quasi-compact and quasi-separated \cite[\href{https://stacks.math.columbia.edu/tag/075S}{Tag 075S}]{stacks-project}. Lastly, a morphism locally of finite type to a locally Noetherian algebraic stack is locally of finite presentation. 
Additionally, if such a morphism is quasi-compact and quasi-separated, then it is of finite presentation. 
See e.g.\ \cite[\href{https://stacks.math.columbia.edu/tag/0DQJ}{Tag 0DQJ}]{stacks-project}.

\subsubsection{Associated triangulated categories}
\label{sec:prelim_stacks_categories}

Denote by $\operatorname{Mod}(\mathcal{X})$ the Grothendieck abelian category of sheaves of $\mathcal{O}_\mathcal{X}$-modules on the lisse-\'{e}tale site of $\mathcal{X}$. 
Set $\operatorname{Qcoh}(\mathcal{X})$ (resp.\ $\operatorname{coh}(\mathcal{X})$) to be the strictly full subcategory of $\operatorname{Mod}(\mathcal{X})$ consisting of quasi-coherent (resp.\ coherent) sheaves. 
Define $D(\mathcal{X}):=D(\operatorname{Mod}(\mathcal{X}))$ as the derived category of $\operatorname{Mod}(\mathcal{X})$. 
Also, $D_{\operatorname{qc}}(\mathcal{X})$ (resp.\ $D^b_{\operatorname{coh}}(\mathcal{X})$) is the full subcategory of $D(\mathcal{X})$ consisting of complexes with quasi-coherent cohomology sheaves (resp.\ which are bounded and with coherent cohomology). 
Lastly, $\operatorname{Perf}(\mathcal{X})$ is the full subcategory of perfect complexes in $D_{\operatorname{qc}}(\mathcal{X})$.

\subsubsection{Affine pointed}
\label{sec:prelim_stacks_affine_pointed}

We say $\mathcal{X}$ is \textbf{affine-pointed} if every morphism $\operatorname{Spec}(k) \to \mathcal{X}$ from a field $k$ is affine. 
For example, any algebraic stack with quasi-affine or quasi-finite diagonal is affine-pointed. See \cite{Hall/Rydh:2019} for details.

\subsubsection{Concentratedness}
\label{sec:prelim_stacks_concentrated}

The following was introduced by \cite[Definition 2.4]{Hall/Rydh:2017}. 
A morphism of algebraic stacks is called \textbf{concentrated} if it is quasi-compact, quasi-separated, and if the derived pushforward of any base change along a quasi-compact quasi-separated morphism has finite cohomological dimension. 
We refer the reader to \cite[\S 2]{Hall/Rydh:2017} for details. One case of morphisms that are concentrated includes those which are representable by algebraic spaces \cite[Lemma 2.5]{Hall/Rydh:2017}. 
An algebraic stack $\mathcal{X}$ is \textbf{concentrated} if it is quasi-compact, quasi-separated, and its structure morphism $\mathcal{X} \to \operatorname{Spec}(\mathbb{Z})$ is concentrated.

\subsubsection{Perfect complexes}
\label{sec:prelim_stacks_perfects}

On any ringed site, e.g.\ lisse-\'{e}tale site of $\mathcal{X}$, the notion of perfect complexes are definable \cite[\href{https://stacks.math.columbia.edu/tag/08G4}{Tag 08G4}]{stacks-project}. 
Particularly, a complex is \textbf{strictly perfect} if it is a bounded complex with each term a direct summand of a finite free, whereas it is \textbf{perfect} if it is locally strictly perfect. 
Denote by $\operatorname{Perf}(\mathcal{X})$ for the triangulated subcategory of $D_{\operatorname{qc}}(\mathcal{X})$ consisting of perfect complexes. 
In general, the compact objects of $D_{\operatorname{qc}}(\mathcal{X})$ are perfect complexes \cite[Lemma 4.4]{Hall/Rydh:2017}, but the converse need not be true. 
This is the case if, and only if, the algebraic stack is concentrated \cite[Remark 4.6]{Hall/Rydh:2017}.

\subsubsection{Supports}
\label{sec:prelim_stacks_supports}

Let $p \colon U \to \mathcal{X}$ be a smooth surjective morphism from a scheme. The notion of support for objects in $D_{\operatorname{qc}}(U)$ extends to that of $\mathcal{X}$ as follows. 
For any $M\in\operatorname{Qcoh}(\mathcal{X})$, set $\operatorname{supp}(M):= p(\operatorname{supp}(p^\ast M))$. 
More generally, given $E\in D_{\operatorname{qc}}(X)$, define the \textbf{support of $E$} as $\operatorname{supp}(E) := \cup_{j\in\mathbb{Z}} \operatorname{supp}(\mathcal{H}^j (E))$. 
It is possible to check that this is independent of the choice of $p$. See \cite[Section 4.3]{Hall/Rydh:2017} for more details. Consider a closed subset $Z\subseteq|\mathcal{X}|$. We say $E\in D_{\operatorname{qc}}(\mathcal{X})$ is \textbf{supported on $Z$} if $\operatorname{supp}(E)\subseteq Z$. Set $D_{\operatorname{qc},Z}(\mathcal{X})$ as the full subcategory of $D_{\operatorname{qc}}(\mathcal{X})$ consisting of objects supported on $Z$. We define similar categories using the adornments $+$, $-$, $b$, etc. Also, we define similar categories on `smaller' objects, e.g.\ $\operatorname{Perf}_Z (\mathcal{X})$ or $D^b_{\operatorname{coh},Z}(\mathcal{X})$.

\subsubsection{Approximation by compacts}
\label{sec:prelim_algebraic_stacks_approx}

The following is a concept motivated by Lipman--Neeman for schemes \cite{Lipman/Neeman:2007}. 
It has been extended to algebraic stacks by \cite{Hall/Lamarche/Lank/Peng:2025}. 
Recall that perfect complexes need not coincide with the compact objects of $D_{\operatorname{qc}}(\mathcal{X})$. 
Consider the datum $(T, E, m)$ where $T\subseteq|\mathcal{X}|$ is closed, $E\in D_{\operatorname{qc}}(\mathcal{X})$, and $m\in \mathbb{Z}$. 
We say \textbf{approximation by compacts} holds for $(T, E, m)$ if there exists a $C\in D_{\operatorname{qc},T}(\mathcal{X})$ compact in $D_{\operatorname{qc}}(\mathcal{X})$ and $C\to E$ such that the induced morphism $\mathcal{H}^i(C) \to \mathcal{H}^i(E)$ is an isomorphism if $i>m$ and is surjective if $i=m$. More generally, one says that $\mathcal{X}$ satisfies \textbf{approximation by compacts} if for every $T\subseteq |\mathcal{X}|$ closed (here, $T$ has quasi-compact complement as $\mathcal{X}$ is Noetherian), there exists an integer $r$ such that for any $(T, E, m)$, where $E$ is $(m-r)$-pseudocoherent (see e.g.\ \cite[\href{https://stacks.math.columbia.edu/tag/08FT}{Tag 08FT}]{stacks-project}) and $\mathcal{H}^{i}(E)$ is supported on $T$ for $i\geq m-r$, 
approximation by compacts holds. 
The reader is referred to \cite[$\S 3$]{Hall/Lamarche/Lank/Peng:2025} for further details. 
This property holds for any algebraic stack with quasi-finite separated diagonal \cite[Corollary 5.4]{Hall/Lamarche/Lank/Peng:2025}.

\subsubsection{Thomason condition}
\label{sec:prelim_algebraic_stacks_thomason}

Let $\beta$ be a cardinal. 
We say $\mathcal{X}$ satisfies the \textbf{$\beta$-Thomason condition} if $D_{\operatorname{qc}}(\mathcal{X})$ is compactly generated by a collection of size $\beta$, and for each quasi-compact open immersion $\mathcal{U}\to\mathcal{X}$, there exists $P\in D_{\operatorname{qc}}(\mathcal{X})^{c}$ such that $\operatorname{supp}(P)=|\mathcal{X}|\setminus|\mathcal{U}|$. 
Any quasi-compact algebraic stack with quasi-finite separated diagonal satisfies this condition \cite[Theorem A]{Hall/Rydh:2017}. 
If a quasi-compact quasi-separated algebraic stack is $\beta$-Thomason, $D_{\operatorname{qc},Z}(\mathcal{X})$ is compactly generated by a collection of size $\beta$ whenever $Z$ is a closed subset with quasi-compact complement \cite[Lemma 4.10]{Hall/Rydh:2017}.

\section{Frobenius morphism}
\label{sec:finiteness_Frobenii}

We discuss the Frobenius morphism on algebraic stacks over $\mathbb{F}_p$. 
See \cite[Definition 3.1.1]{Olsson:2007}, \cite{Zhang:2024}, \cite{Castorena/Neumann:2024} for other variations.

\begin{reminder}
    Let $\mathcal{X}$ be an algebraic stack over $\mathbb{F}_p$. 
    Define a morphism $F_{\mathcal{X}} \colon \mathcal{X} \to \mathcal{X}$ of fibered categories over $(\operatorname{Sch}_{\mathbb{F}_p})_{fppf}$ as follows. 
    We first describe $F_{\mathcal{X}}$ on objects. 
    Let $T$ be a $\mathbb{F}_p$-scheme over $\mathcal{X}$. Let $x$ be an object of $\mathcal{X}(T)$, i.e. $x\colon T\to \mathcal{X}$ is a morphism of stacks. Note that $T$ admits a Frobenius morphism $F_T$. We define $F_\mathcal{X}(x)$ to be the composition$$T\overset{F_T}{\longrightarrow} T\overset{x}{\longrightarrow}\mathcal{X}.$$

    Next, we describe $F_{\mathcal{X}}$ on morphisms. Let $\phi\colon T^\prime \to T$ be a morphism of $\mathbb{F}_p$ schemes. We have the following commutative diagram
    \begin{displaymath}
        \begin{tikzcd}
            {T^\prime} & {T^\prime} \\
            && {\operatorname{Spec}(\mathbb{F}_p)} \\
            {T} & {T}
            \arrow["{F_{T^\prime}}", from=1-1, to=1-2]
            \arrow["{\phi}"', from=1-1, to=3-1]
            \arrow["{\pi_{T^\prime}}", from=1-2, to=2-3]
            \arrow["{\phi}", from=1-2, to=3-2]
            \arrow["{F_{T}}"', from=3-1, to=3-2]
            \arrow["{\pi_T}"', from=3-2, to=2-3]
        \end{tikzcd}
    \end{displaymath}
    where $\pi_{\#}$ are the structure morphisms. In particular, there is a canonical isomorphism $\phi^*F_{T^\prime}^*\cong F_T^*\phi^*$. We can then define $F_\mathcal{X}(\phi)$ to be $\phi$ itself. By the 2-Yoneda lemma, we see that $F_\mathcal{X}$ is a morphism of stacks as desired.
\end{reminder}

\begin{definition}
    \label{def:absolute_Frobenius}
    The morphism $F_{\mathcal{X}} \colon \mathcal{X} \to \mathcal{X}$ is called the \textbf{absolute Frobenius morphism} on $\mathcal{X}$. 
    Its $e$-th iterate is denoted by $F^e_{\mathcal{X}}$, i.e.
    \begin{displaymath}
        F^e_{\mathcal{X}} = \underbrace{F_{\mathcal{X}} \circ \cdots \circ F_{\mathcal{X}}}_{e \textrm{ times}}.
    \end{displaymath}
    We often omit `absolute' and the subscript if it is clear from context. 
    If $\mathcal{Y} \to \mathcal{X}$ is a morphism of algebraic stacks, then we write $F_{\mathcal{Y}/\mathcal{X}}$ for the \textbf{relative Frobenius morphism of $\mathcal{Y}/\mathcal{X}$}; that is, the unique morphism which fits into the following diagram,
    \begin{displaymath}
        \begin{tikzcd}
            {\mathcal{Y}} & {\mathcal{X}\times_{F,\mathcal{X}} \mathcal{Y}} & {\mathcal{Y}} \\
            & {\mathcal{X}} & {\mathcal{X}.}
            \arrow["{F_{\mathcal{Y}/\mathcal{X}}}"', from=1-1, to=1-2]
            \arrow["F"', bend right =-30pt, from=1-1, to=1-3]
            \arrow[bend right =12pt, from=1-1, to=2-2]
            \arrow[from=1-2, to=1-3]
            \arrow[from=1-2, to=2-2]
            \arrow[from=1-3, to=2-3]
            \arrow["F"', from=2-2, to=2-3]
        \end{tikzcd}
    \end{displaymath}
\end{definition}

\begin{remark}
    \label{rm:etale_base_change_Frobenii}
    Let $g\colon \mathcal{Y} \to \mathcal{X}$ be an \'{e}tale morphism of Noetherian algebraic stacks over $\mathbb{F}_p$. 
    If $g$ is representable by algebraic spaces (e.g.\ an open immersion), then $F_{\mathcal{Y}/\mathcal{X}}$ is an isomorphism and the diagram
    \begin{displaymath}
        \begin{tikzcd}
            {\mathcal{Y}} & {\mathcal{Y}} \\
            {\mathcal{X}} & {\mathcal{X}}
            \arrow["F_{\mathcal{Y}}"', from=1-1, to=1-2]
            \arrow[from=1-1, to=2-1]
            \arrow["F_{\mathcal{X}}"', from=2-1, to=2-2]
            \arrow["g"', from=1-2, to=2-2]
        \end{tikzcd}
    \end{displaymath}
    is Cartesian. 
    Indeed, as $g$ is representable by algebraic spaces, we may assume that $\mathcal{X}$ is a scheme and $\mathcal{Y}$ is an algebraic space. By choosing an \'etale presentation for $\mathcal{Y}$, we may further assume that $\mathcal{Y}$ is also a scheme. In this case, the desired isomorphism holds by \cite[\href{https://stacks.math.columbia.edu/tag/0EBS}{Tag 0EBS}]{stacks-project}.
\end{remark}

\begin{proposition}
    [Rydh]
    \label{lem:homeomorphism_on_underlying_topological_space}
    Let $\mathcal{X}$ be an algebraic stack over $\mathbb{F}_p$. 
    Then the Frobenius morphism on $\mathcal{X}$ is a separated universal homeomorphism in the sense of \cite[\href{https://stacks.math.columbia.edu/tag/0CI7}{Tag 0CI7}]{stacks-project}. 
    In particular, the diagonal of $F_\mathcal{X}$ is finite.
\end{proposition}

\begin{proof}
    Choose a smooth presentation $U\to\mathcal{X}$. Consider the following diagram
    \begin{displaymath}
        \begin{tikzcd}
            {U} & {\mathcal{X}\times_{F,\mathcal{X}} U} & {U} \\
            & {\mathcal{X}} & {\mathcal{X}.}
            \arrow["{F_{U/\mathcal{X}}}"', from=1-1, to=1-2]
            \arrow["F_{\mathcal{X}}"', bend right =-30pt, from=1-1, to=1-3]
            \arrow[bend right =12pt, from=1-1, to=2-2]
            \arrow[from=1-2, to=1-3]
            \arrow[from=1-2, to=2-2]
            \arrow[from=1-3, to=2-3]
            \arrow["F"', from=2-2, to=2-3]
        \end{tikzcd}
    \end{displaymath}
    We know that $F_U$ is a universal homeomorphism \cite[\href{https://stacks.math.columbia.edu/tag/0CC8}{Tag 0CC8}]{stacks-project}. 
    By \cite[Lemma 3.2.4]{Olsson:2007}, the relative Frobenius morphism $F_{U/\mathcal{X}}$ is a finite flat covering. 
    Moreover, it is also universally injective. It follows that $F_{U/\mathcal{X}}$ is a universal homeomorphism. 
    By \cite[\href{https://stacks.math.columbia.edu/tag/0H2M}{Tag 0H2M}]{stacks-project}, the projection $\mathcal{X}\times_{F,\mathcal{X}} U\to U$ is a universal homeomorphism.
    It follows from \cite[\href{https://stacks.math.columbia.edu/tag/0DTQ}{Tag 0DTQ}]{stacks-project} that $F_\mathcal{X}$ is a universal homeomorphism.
    
    To see that $F_\mathcal{X}$ has finite diagonal, we first observe that the relative diagonal $\Delta_{F_\mathcal{X}}$ is the same as the relative Frobenius morphism of the diagonal $\Delta_\mathcal{X}$ of $\mathcal{X}$. Indeed, we have the following fibered square
    \begin{displaymath}
        \begin{tikzcd}
            {\mathcal{X}} & {\mathcal{X}\times_{F_\mathcal{X},\mathcal{X}}\mathcal{X}} & {\mathcal{X}\times\mathcal{X}} \\
            & {\mathcal{X}} & {\mathcal{X}\times\mathcal{X}.}
            \arrow["\Delta_{F_\mathcal{X}}"', from=1-1, to=1-2]
            \arrow["F_{\mathcal{X}\times\mathcal{X}}"', bend right =-30pt, from=1-1, to=1-3]
            \arrow[bend right =12pt, from=1-1, to=2-2]
            \arrow[from=1-2, to=1-3]
            \arrow[from=1-2, to=2-2]
            \arrow["F_\mathcal{X}\times F_\mathcal{X}", from=1-3, to=2-3]
            \arrow["\Delta_\mathcal{X}"', from=2-2, to=2-3]
        \end{tikzcd}
    \end{displaymath}
    Using the same argument above, we see that $\Delta_{F_\mathcal{X}}$ is a universal homeomorphism. 
    As $\Delta_{F_\mathcal{X}}$ is representable, \cite[\href{https://stacks.math.columbia.edu/tag/04DF}{Tag 04DF}]{stacks-project} implies it is integral.
    Therefore, $\Delta_{F_\mathcal{X}}$ is integral and locally of finite type, and so it must be finite.
\end{proof}

\begin{remark}
    The fact that $\Delta_{F_\mathcal{X}}$ is a universal homeomorphism was previously known due to Lei Zhang \cite[Lemma 3.1]{Zhang:2024}.
\end{remark}

\begin{example}
    [Hall]
    \label{ex:nonrepresentability}
    Consider the classifying stack $B\mathbb{G}_{m}$ over $\mathbb{F}_p$. 
    The absolute Frobenius morphism is not representable because $B\mathbb{G}_{m}\times_{B\mathbb{G}_m, F_{\mathbb{G}_m}}\mathbb{F}_p$ is isomorphic to $B\mu_p$ (which is an algebraic stac). 
    To see this, let $X$ be a scheme over $\mathbb{F}_p$. 
    A morphism $X\to B\mathbb{G}_m$ gives us a $\mathbb{G}_m$-torsor, which corresponds uniquely, up to isomorphism, to a line bundle $\mathcal{L}$ on $X$. 
    Let $F_X$ be the Frobenius morphism on $X$. 
    By definition of the Frobenius morphism on $B\mathbb{G}_m$, we see that $F_{B\mathbb{G}_m}(X)$ is given by sending $\mathcal{L}$ to $F_{X}^{\ast}\mathcal{L}\cong L^{\otimes p}$. 
    Denote $\operatorname{Spec}(\mathbb{F}_p)\to B\mathbb{G}_m$ the morphism corresponding to the trivial line bundle. 
    Then the objects of $B\mathbb{G}_{m}\times_{B\mathbb{G}_m, F_{\mathbb{G}_m}}\mathbb{F}_p(X)$ correspond to line bundles $\mathcal{L}$ on $X$ with a trivialization $\mathcal{L}^{\otimes p}\cong\mathcal{O}_X$. 
    It is well-known that this is precisely the groupoid of $\mu_p$-torsors over $X$ in the fppf topology. 
    Since $X$ is arbitrary, we see that $B\mathbb{G}_{m}\times_{B\mathbb{G}_m, F_{\mathbb{G}_m}}\mathbb{F}_p$ is isomorphic to $B\mu_p$ over $\mathbb{F}_p$. 
    More generally, let $G$ be an algebraic group (of finite type) over $\mathbb{F}_p$ and $BG$ be the corresponding classifying stack. 
    Then the Frobenius morphism on $BG$ is representable if, and only if, the Frobenius morphism on $G$ is a monomorphism \cite[Warning\ 3.1.3]{Olsson:2007}. However, this only occurs when $G$ is \'{e}tale.
\end{example}

\begin{corollary}
    \label{cor:Frobenius_for_DM_is_representable}
    Let $\mathcal{X}$ be an algebraic stack over $\mathbb{F}_p$. 
    Then the absolute Frobenius morphism $F_\mathcal{X}\colon\mathcal{X}\to\mathcal{X}$ is representable by algebraic spaces if, and only if, $\mathcal{X}$ is Deligne--Mumford. 
\end{corollary}

\begin{proof}
    We first observe that $F_\mathcal{X}\colon \mathcal{X} \to \mathcal{X}$ is representable by algebraic spaces if, and only if, the induced morphism on the stabilizer groups $G_x \to G_{F(x)}$ is a monomorphism for every geometric point $x$ of $\mathcal{X}$. 
    On the other hand, the Frobenius morphism $F\colon \mathcal{X} \to \mathcal{X}$ induces a morphism $G_x \to G_{F(x)}$, which coincides with the relative Frobenius morphism of $G_x$ over its residual field (see \Cref{def:absolute_Frobenius}). It suffices to show that $\mathcal{X}$ is Deligne--Mumford if, and only if, the relative Frobenius morphism of $G_x$ is a monomorphism for every geometric point $x$ of $\mathcal{X}$. 
    However, the relative Frobenius morphism of $G_x$ is a monomorphism if, and only if, $G_x$ is \'etale over its residue field. Thus, the claim follows.
\end{proof}

\section{\texorpdfstring{$F$}{F}-finiteness}
\label{sec:finiteness_F-finiteness}

We propose an extension of $F$-finiteness from schemes to algebraic stacks over $\mathbb{F}_p$.

\begin{definition}
    \label{def:F_finite}
    An algebraic stack $\mathcal{X}$ over $\mathbb{F}_p$ is called \textbf{$F$-finite} if the Frobenius morphism $F_\mathcal{X}\colon \mathcal{X} \to \mathcal{X}$ is concentrated and locally of finite type. 
\end{definition}

\begin{remark}
    \label{rmk:F-finiteness}
    \Cref{def:F_finite} generalizes the notion of $F$-finiteness for schemes. Indeed, the Frobenius morphism on a scheme is always an integral morphism of schemes. 
    Moreover, every quasi-compact quasi-separated morphism of schemes is concentrated, whereas every locally of finite type integral morphism is finite. 
\end{remark}

\begin{lemma}
    \label{lem:F-finiteness_ascent_descent}
    Let $f\colon\mathcal{Y}\to\mathcal{X}$ be a morphism of algebraic stacks.
    \begin{enumerate}
        \item If $\mathcal{X}$ is $F$-finite (resp.\ $F$-finite and $F_\mathcal{X}$ is representable) and $f$ is concentrated (resp.\ representable) and locally of finite type, then $\mathcal{Y}$ is $F$-finite (resp.\ $F$-finite and $F_\mathcal{Y}$ is representable).
        \item If $\mathcal{X}$ is concentrated, $\mathcal{Y}$ is $F$-finite and $f$ is flat, surjective, representable, and locally of finite presentation, then $\mathcal{X}$ is $F$-finite.
    \end{enumerate}
\end{lemma}

\begin{proof}
    Consider the commutative diagram
    \begin{displaymath}
        \begin{tikzcd}
            {\mathcal{Y}} & {\mathcal{X}\times_{F,\mathcal{X}} \mathcal{Y}} & {\mathcal{Y}} \\
            & {\mathcal{X}} & {\mathcal{X}.}
            \arrow["{F_{\mathcal{Y}/\mathcal{X}}}"', from=1-1, to=1-2]
            \arrow["F_{\mathcal{Y}}"', bend right =-30pt, from=1-1, to=1-3]
            \arrow["f",bend right =12pt, from=1-1, to=2-2]
            \arrow[from=1-2, to=1-3]
            \arrow["f^\prime",from=1-2, to=2-2]
            \arrow[from=1-3, to=2-3]
            \arrow["F_{\mathcal{X}}"', from=2-2, to=2-3]
        \end{tikzcd}
    \end{displaymath}
    Suppose $\mathcal{X}$ is $F$-finite. 
    It suffices to show $F_{\mathcal{Y}/\mathcal{X}}$ is concentrated and locally of finite type. 
    By assumption, both $f$ and $f^\prime$ are concentrated and locally of finite type. 
    The result then follows from \cite[Lemma 2.5(4)]{Hall/Rydh:2017} and \cite[\href{https://stacks.math.columbia.edu/tag/06U9}{Tag 06U9}]{stacks-project}. 

    Suppose $\mathcal{Y}$ is $F$-finite. 
    The same argument shows that $F_{\mathcal{Y}/\mathcal{X}}$ is concentrated and locally of finite type. 
    Applying the argument once more to $F_\mathcal{Y}$ and $F_{\mathcal{Y}/\mathcal{X}}$ gives that $\mathcal{X}\times_{F,\mathcal{X}} \mathcal{Y}\to\mathcal{Y}$ is concentrated and locally of finite type. 
    The statement then follows from \cite[Lemma 2.5(2)]{Hall/Rydh:2017} and \cite[\href{https://stacks.math.columbia.edu/tag/06U7}{Tag 06U7}]{stacks-project}.
\end{proof}

\begin{remark}
    In particular, \Cref{lem:F-finiteness_ascent_descent} implies that $F$-finiteness is local in the fppf topology for concentrated algebraic stacks. This is false in general. Indeed, consider the smooth covering map $\operatorname{Spec}(\mathbb{F}_p)\to B\mathbb{G}_a$. The field $\mathbb{F}_p$ is certainly $F$-finite. It turns out that $B\mathbb{G}_a$ is not $F$-finite \Cref{ex:Ga_not_F-finite}.
\end{remark}

\begin{proposition}
    \label{prop:F-finite_via_finite_presentation}
    Let $\mathcal{X}$ be an algebraic stack over $\mathbb{F}_p$. 
    Denote by $F_{\mathcal{X}}\colon \mathcal{X} \to \mathcal{X}$ the Frobenius morphism of $\mathcal{X}$. 
    Then the following are equivalent:
    \begin{enumerate}
        \item \label{prop:F-finite_via_finite_presentation1} $\mathcal{X}$ is $F$-finite
        \item \label{prop:F-finite_via_finite_presentation2} there is a smooth surjective morphism $U \to \mathcal{X}$ from an $F$-finite scheme and $\mathcal{X}\times_{F_\mathcal{X},\mathcal{X}} U$ is concentrated.
        \item \label{prop:F-finite_via_finite_presentation3} $U$ is $F$-finite and $\mathcal{X}\times_{F_\mathcal{X},\mathcal{X}} U$ is concentrated for every smooth surjective morphism $U \to \mathcal{X}$.
    \end{enumerate}
    If $\mathcal{X}$ is locally Noetherian, then these statements are equivalent to
    \begin{enumerate}
        \setcounter{enumi}{3}
        \item \label{prop:F-finite_via_finite_presentation4}$F_\mathcal{X}$ is proper and $F_{\mathcal{X},\ast}\colon\operatorname{QCoh}(\mathcal
        X)\to\operatorname{QCoh}(\mathcal
        X)$ is exact.
    \end{enumerate}
\end{proposition}

\begin{proof}
    From \Cref{lem:F-finiteness_ascent_descent}, $\eqref{prop:F-finite_via_finite_presentation1} \implies \eqref{prop:F-finite_via_finite_presentation3}$. 
    It is clear that $\eqref{prop:F-finite_via_finite_presentation3} \implies \eqref{prop:F-finite_via_finite_presentation2}$. 
    For $\eqref{prop:F-finite_via_finite_presentation2} \implies \eqref{prop:F-finite_via_finite_presentation1}$. 
    Assume there is a smooth surjective morphism $s\colon U \to \mathcal{X}$ from an $F$-finite scheme such that $\mathcal{X}\times_{F_\mathcal{X},\mathcal{X}} U$ is concentrated. 
    It follows from \cite[Lemma 2.5(4)]{Hall/Rydh:2017} that the projection $\mathcal{X}\times_{F_\mathcal{X},\mathcal{X}} U\to U$ is also concentrated. 
    By \cite[Lemma 2.5(2)]{Hall/Rydh:2017}, we see that $F_\mathcal{X}$ is concentrated. 
    Moreover, the morphisms $F_U$ and $F_{U/\mathcal{X}}$ are both locally of finite type. 
    So is the projection $\mathcal{X}\times_{F_\mathcal{X},\mathcal{X}} U\to U$. 
    By descent, we have that $F_\mathcal{X}$ is locally of finite type as desired.

    Now suppose $\mathcal{X}$ is locally Noetherian. 
    We first show $\eqref{prop:F-finite_via_finite_presentation2} \implies \eqref{prop:F-finite_via_finite_presentation4}$. 
    Suppose there is a smooth surjective morphism $U \to \mathcal{X}$ from an $F$-finite scheme and $\mathcal{X}\times_{F_\mathcal{X},\mathcal{X}} U$ is concentrated. 
    Since $U$ is $F$-finite, we observed that the projection $\mathcal{X}\times_{F_\mathcal{X},\mathcal{X}} U\to U$ is locally of finite type. 
    By \Cref{lem:homeomorphism_on_underlying_topological_space}, we see that the diagonal of $F_\mathcal{X}$ is finite, and hence that of $\mathcal{X}\times_{F_\mathcal{X},\mathcal{X}} U\to U$ as well. 
    Applying \cite[Theorem 6.12]{Rydh:2013}, we can find a coarse moduli space $\pi\colon \mathcal{X}\times_{F_\mathcal{X},\mathcal{X}} U\to X$ over $U$ such that $\pi$ is a proper universal homeomorphism and $X$ is an algebraic space that is locally of finite type over $U$. 
    Since $\mathcal{X}\times_{F_\mathcal{X},\mathcal{X}} U$ is concentrated and locally Noetherian, $\pi_\ast$ is exact and preserves coherence. 
    Moreover, we know that $\mathcal{X}\times_{F_\mathcal{X},\mathcal{X}} U\to U$ is a universal homeomorphism. 
    Hence, $X\to U$ is a universal homeomorphism from an algebraic space that is locally of finite type, and thus must be finite. 
    It follows that $F_\mathcal{X}$ is proper and $F_{\mathcal{X},\ast}$ is exact.
    
    It suffices to show $\eqref{prop:F-finite_via_finite_presentation4} \implies \eqref{prop:F-finite_via_finite_presentation3}$. 
    Suppose $F_\mathcal{X}$ is proper and $F_{\mathcal{X},\ast}$ is exact. Choose an arbitrary smooth presentation $U\to\mathcal{X}$.
    It follows that $p_2\colon \mathcal{X}\times_{F_\mathcal{X},\mathcal{X}} U\to U$ is also proper and $p_{2,\ast}$ is exact. By \cite[Lemma 3.2.4]{Olsson:2007}, the relative Frobenius $F_{U/\mathcal{X}}$ is finite. It follows that $F_U$ is proper and thus $U$ is $F$-finite.
    By \cite[Theorem 6.12]{Rydh:2013}, the projection $p_2$ factors through the coarse moduli space
    \begin{displaymath}
    \mathcal{X}\times_{F_\mathcal{X},\mathcal{X}} U\overset{\pi}{\longrightarrow} X\overset{p_2^\prime}{\longrightarrow} U
    \end{displaymath}
    where $\pi$ is a separated universal homeomorphism and $X$ is an algebraic space over $U$ and $\mathcal{O}_X\cong\pi_\ast\mathcal{O}_{\mathcal{X}\times_{F_\mathcal{X},\mathcal{X}} U}$. Since $p_2$ is quasi-compact, quasi-separated, and $p_{2,*}$ is exact, so is $\pi$. It follows that $\mathcal{X}\times_{F_\mathcal{X},\mathcal{X}} U$ is concentrated by \cite[Theorem 3.6]{Abramovich/Olsson/Vistoli:2008}. This finishes the proof.
\end{proof}

\begin{example}
    [van Dobben de Bruyn--Hartl]
    \label{ex:DM_F-finite}
    Let $X$ be an $F$-finite scheme. 
    Any Deligne--Mumford stack that is locally of type over $X$ is $F$-finite. This follows from \Cref{cor:Frobenius_for_DM_is_representable} and \Cref{prop:F-finite_via_finite_presentation}.
\end{example}

\begin{example}
    [Infinite stabilizers]
    \label{ex:Gm_F-finite}
    Consider the classifying stack $B\mathbb{G}_m$ over $\mathbb{F}_p$. By \Cref{ex:nonrepresentability}, the Frobenius morphism on $B\mathbb{G}_m$ is not representable. 
    However, it is $F$-finite because $B\mu_p$ is concentrated and proper over $\mathbb{F}_p$. 
    This is because $\mu_p$ is a finite linearly reductive algebraic group over $\mathbb{F}_p$. 
    Therefore, we see that $B\mathbb{G}_m$ is $F$-finite. 
\end{example}

\begin{example}
    [Necessity of concentrated]
    \label{ex:Ga_not_F-finite}
    Consider the classifying stack $B\mathbb{G}_a$ over $\mathbb{F}_p$. 
    The fibre product of $F_{B\mathbb{G}_a}$, with the natural covering $\operatorname{Spec}(\mathbb{F}_p)\to B\mathbb{G}_a$, is precisely $B\alpha_p$. 
    It is still proper but not concentrated over $\mathbb{F}_p$ because $\alpha_p$ is not linearly reductive in characteristic $p$. 
    Therefore, the Frobenius morphism on $B\mathbb{G}_a$ is not $F$-finite. 
\end{example}

\begin{remark}
    One could remove the concentrated condition in \Cref{def:F_finite}, and instead require the absolute Frobenius morphism to be locally of finite type. 
    In this case, stacks such as $B\mathbb{G}_a$ and $B\alpha_p$ would be $F$-finite. 
    However, the Frobenius pushforward would no longer be exact on these stacks, as demonstrated in \Cref{ex:Ga_not_F-finite}. 
    On the other hand, $F$-finite algebraic stacks need not be concentrated. Indeed, a Deligne--Mumford stack locally of finite type over $\mathbb{F}_p$ need not be concentrated (e.g.\ $B\mathbb{Z}/p\mathbb{Z}$) but must be $F$-finite by \Cref{ex:DM_F-finite}. 
\end{remark}

\begin{proposition}
    \label{lem:F-finite_classifying_stack}
    Let $k$ be a field of characteristic $p>0$ that is $F$-finite. 
    Let $G$ be a group algebraic space locally of finite type over $k$. 
    Let $\overline{G}=G\times_{k}\overline{k}$. 
    Consider the following statements:
    \begin{enumerate}
        \item\label{lem:F-finite_classifying_stack1} The classifying stack $BG$ is $F$-finite.
        \item\label{lem:F-finite_classifying_stack2} The Frobenius kernel of $G$ is linearly reductive.
        \item\label{lem:F-finite_classifying_stack3} The identity component $\overline{G}^\circ$ is diagonalizable.
    \end{enumerate}
    Then we have $\eqref{lem:F-finite_classifying_stack1}\iff\eqref{lem:F-finite_classifying_stack2}\impliedby\eqref{lem:F-finite_classifying_stack3}$. 
    If $G$ is also affine, then $\eqref{lem:F-finite_classifying_stack2}\implies\eqref{lem:F-finite_classifying_stack3}$.
\end{proposition}

\begin{proof}
    For $\eqref{lem:F-finite_classifying_stack1}\iff\eqref{lem:F-finite_classifying_stack2}$, let $K$ be the Frobenius kernel of $G$. We see that $BG\times_{F_{BG},BG}\operatorname{Spec}(k)\cong BK$. 
    By \Cref{prop:F-finite_via_finite_presentation}, $BG$ is $F$-finite if, and only if, $BK$ is concentrated and proper. 
    Applying \cite[Theorem B]{Hall/Rydh:2015}, the latter is equivalent to the statement that $K$ is linearly reductive. 
    In particular, $K$ is finite over $k$.

    We check that $\eqref{lem:F-finite_classifying_stack3}\implies\eqref{lem:F-finite_classifying_stack2}$.
    Suppose $\overline{G}^\circ$ is diagonalizable. 
    Note that the Frobenius kernel $\overline{K}$ of $\overline{G}$ is a closed subgroup of $\overline{G}^\circ$. 
    It follows that the Frobenius kernel is also diagonalizable and thus linearly reductive. 
    Descending along the field extension $\overline{k}/k$, we see that the Frobenius kernel $K$ of $G$ is linearly reductive.
    
    For $\eqref{lem:F-finite_classifying_stack2}\implies\eqref{lem:F-finite_classifying_stack3}$, suppose that $G$ is also affine. 
    We claim that $G^\circ$ does not contain $\alpha_p$ as a subgroup. 
    Indeed, we first note that the Frobenius kernel $K$ is connected and thus contained in $G^\circ$. Since $K$ is linearly reductive, it does not contain $\alpha_p$ as a subgroup. 
    This implies that $G^\circ$ does not contain $\alpha_p$ as a subgroup because $\alpha_p$ is annihilated by the Frobenius map and therefore must be contained in $K$. 
    Since $G^\circ$ is a connected affine algebraic group over $k$, it is of multiplicative type. 
    Thus, $\overline{G}^\circ$ is diagonalizable. 
\end{proof}

\begin{remark}
    \label{rmk:F-finite_fields}
    A field $k$ of characteristic $p>0$ is $F$-finite if and only if $[k:k^p]$ is finite. For example, every perfect field of characteristic $p>0$ is $F$-finite. For an imperfect example, we see that $\mathbb{F}_p(x_1,\ldots,x_n)$ is $F$-finite. However, the field $\mathbb{F}_p(x_1,x_2,\ldots)$ has infinite transcendence degree and thus cannot be $F$-finite. In particular, \Cref{lem:F-finite_classifying_stack} is false over non-$F$-finite fields like $\mathbb{F}_p(x_1,x_2,\ldots)$. Indeed, the trivial group $G=\operatorname{Spec}(\mathbb{F}_p(x_1,x_2,\ldots))$ is linearly reductive but not $F$-finite. 
\end{remark}

\begin{example}
    \label{ex:F-finite_elliptic_curves}
    The affine assumption on $G$ is necessary in \Cref{lem:F-finite_classifying_stack}. 
    To see this, let $E$ be an elliptic curve over $\overline{\mathbb{F}_p}$. 
    Let $\overline{E}$. 
    It is well-known that $E$ is ordinary (resp.\ supersingular) if, and only if, the Frobenius kernel of $E$ is isomorphic to $\mu_p$ (resp.\ $\alpha_p$). 
    By \Cref{lem:F-finite_classifying_stack}, $BE$ is $F$-finite if, and only if, $E$ is ordinary. 
    In particular, this tells us that $F$-finite algebraic stacks need not have affine stabilizers.
\end{example}

\begin{remark}
    $F$-finite algebraic stacks need not be concentrated even with the added affine stabilizers assumption. 
    Recall from \cite[Definition 1.1]{Hall/Rydh:2015} that an affine group scheme $G$ over a field $k$ of characteristic $p>0$ is nice if $G^\circ$ is of multiplicative type and the order of its component group $\pi_0(G)$ is coprime to $p$. 
    \Cref{lem:F-finite_classifying_stack} implies that $BG$ is $F$-finite if $G$ is nice. 
    However, the converse is not true. 
    This is because the order of $\pi_0(G)$ need not be coprime to $p$ for $BG$ to be $F$-finite. 
    Therefore, $F$-finite algebraic stacks need not be concentrated by \cite[Theorem B \& C]{Hall/Rydh:2015}.
\end{remark}

We conclude this section with an example of an $F$-finite algebraic stack that is neither Deligne--Mumford nor concentrated. 

\begin{example}
    \label{ex:F-finite_non_DM_non_tame}
    Let $G=\mathbb{G}_m\times\mathbb{Z}/p\mathbb{Z}$. 
    Consider the classifying stack $BG$ over $\mathbb{F}_p$. 
    We see that $G$ is not an \'{e}tale group scheme over $\mathbb{F}_p$, and thus $BG$ is not a Deligne--Mumford stack. 
    Moreover, $G$ is not linearly reductive over $\mathbb{F}_p$. 
    By \cite[Theorem B and C]{Hall/Rydh:2015}, the stack $BG$ is not concentrated.
    We claim that $BG$ is $F$-finite. 
    To see this, we first observe that $F_{BG}$ is of finite type since $BG$ is of finite type over $\mathbb{F}_p$. 
    It suffices to show $F_{BG}$ is concentrated. 
    By \Cref{lem:F-finite_classifying_stack}, it suffices to show the Frobenius kernel of $G$ is linearly reductive. 
    But we know the Frobenius kernel of $\mathbb{G}_m$ is $\mu_p$ by \Cref{ex:Gm_F-finite} and that of $\mathbb{Z}/p\mathbb{Z}$ is trivial by \Cref{rm:etale_base_change_Frobenii} since $\mathbb{Z}/p\mathbb{Z}$ is \'{e}tale over $\mathbb{F}_p$. 
    Therefore, we see that the Frobenius kernel of $G$ is $\mu_p$ and the claim follows. 
\end{example}

\section{Codepth}
\label{sec:Frobenius_generation_Variation_of_codepth}

We study numerical conditions that are later used to detect Frobenius generation. 

\begin{definition}
    \hfill
    \begin{enumerate}
        \item Consider a Noetherian ring $R$. 
        For any finitely generated $R$-module $M$, set $\beta^R (M)$ to be the minimal number of elements required to generate $M$.
        By abuse of notation, define $\beta^X (E) := \beta^R (H^0(X,E))$ where $X=\operatorname{Spec}(R)$ and $E$ is a coherent $\mathcal{O}_X$-module.
        \item If $R$ is a local ring with maximal ideal $\mathfrak{m}$, then we set $\operatorname{codepth}(R) := \beta^R (\mathfrak{m}) - \operatorname{depth}(R)$. 
        More generally, if $X$ is a Noetherian scheme, we set $\operatorname{codepth}(X)$ to be the supremum of $\operatorname{codepth}(\mathcal{O}_{X,p})$ indexed over $p\in X$.
    \end{enumerate}
\end{definition}

\begin{remark}
    \label{rmk:codepth_facts}
    If $f\colon \operatorname{Spec}(S) \to \operatorname{Spec}(R)$ is a  morphism of affine Noetherian schemes and $E$ is a coherent $\mathcal{O}_{\operatorname{Spec}(R)}$-module, then $\beta^{\operatorname{Spec}(R)} (E) \geq \beta^{\operatorname{Spec}(S)} (f^\ast E)$. 
    Indeed, $f^\ast$ is right exact. 
    If $(R,\mathfrak{m},k)$ is an $F$-finite local ring, then $\beta^R (F_\ast^e R) <\infty$. 
    In fact, an argument similar to \cite[Proposition 1.1]{Kunz:1976} shows that $\beta^R (F_\ast^e R) \geq \beta^R (\mathfrak{m})$. 
    Indeed, let $\mathfrak{m}^{[p^e]}$ be the ideal generated by $r^{p^e}$ for all $r\in \mathfrak{m}$. 
    Then for all $e\geq 1$,
    \begin{displaymath}
        \begin{aligned}
            \beta^R (F_\ast^e R) 
            &\geq \beta^R (F_\ast R) 
            = \operatorname{rank}_k (k\otimes^R F_\ast R)
            = \operatorname{rank}_k (F_\ast (R/\mathfrak{m}^{[p]}))
            \\&\geq \operatorname{rank}_k (F_\ast k) \operatorname{rank}_k (R/\mathfrak{m}^{[p]}) \geq \operatorname{rank}_k (\mathfrak{m}/\mathfrak{m}^2) = \beta^R (\mathfrak{m}).
        \end{aligned}
    \end{displaymath}
    Consequently, we obtain $\operatorname{codepth}(R) \leq \beta^R (F_\ast^e R) - \operatorname{depth}(R)$ for all $e\geq 1$.
\end{remark}

\begin{lemma}
    \label{lem:bound_along_etale}
    Let $f\colon Y \to X$ be an \'{e}tale morphism of $F$-finite schemes. 
    Then for every $e>0$,
    \begin{displaymath}
        \sup_{s\in X} \left\{ \beta^{\mathcal{O}_{X,s}} (F_\ast^e \mathcal{O}_{X,s}) - \operatorname{depth}(\mathcal{O}_{X,s})  \right\} \geq \sup_{t\in Y} \left\{ \beta^{\mathcal{O}_{Y,t}} (F_\ast^e \mathcal{O}_{Y,t}) - \operatorname{depth}(\mathcal{O}_{Y,t}) \right\}.
    \end{displaymath}
\end{lemma}

\begin{proof}
    As $f$ is \'{e}tale, the induced morphism on local rings $\mathcal{O}_{X,f(t)} \to \mathcal{O}_{Y,t}$ is \'{e}tale. 
    Denote by $f_t\colon \operatorname{Spec}(\mathcal{O}_{Y,t}) \to \operatorname{Spec}(\mathcal{O}_{X,f(t)})$ for the associated morphism of $\mathcal{O}_{X,f(t)} \to \mathcal{O}_{Y,t}$ on affine spectra. 
    By flatness of $f_t$, we know that
    \begin{displaymath}
        \beta^{\mathcal{O}_{X,f(t)}} (F_\ast^e \mathcal{O}_{X,f(t)}) \geq \beta^{\mathcal{O}_{Y,t}} (f_t^\ast F_\ast^e \mathcal{O}_{X,f(t)}).
    \end{displaymath}
    However, $f_t$ being \'{e}tale allows us to use \Cref{rm:etale_base_change_Frobenii}. 
    In particular, 
    \begin{displaymath}
        f_t^\ast F_\ast^e \mathcal{O}_{X,f(t)} \cong F_\ast^e f_t^\ast \mathcal{O}_{X,f(t)} \cong F_\ast^e \mathcal{O}_{Y,t},
    \end{displaymath}
    and so,
    \begin{displaymath}
        \beta^{\mathcal{O}_{X,f(t)}} (F_\ast^e \mathcal{O}_{X,f(t)}) \geq \beta^{\mathcal{O}_{Y,t}} (F_\ast^e \mathcal{O}_{Y,t}).
    \end{displaymath}
    Furthermore, \'{e}taleness of $f_t$ implies $\operatorname{depth}(\mathcal{O}_{Y,t}) = \operatorname{depth}(\mathcal{O}_{X,f(t)})$ (see e.g.\ \cite[\href{https://stacks.math.columbia.edu/tag/039T}{Tag 039T}]{stacks-project}). 
    So, it follows that
    \begin{displaymath}
        \beta^{\mathcal{O}_{X,f(t)}} (F_\ast^e \mathcal{O}_{X,f(t)}) - \operatorname{depth}(\mathcal{O}_{X,f(t)}) \geq \beta^{\mathcal{O}_{Y,t}} (F_\ast^e \mathcal{O}_{Y,t}) - \operatorname{depth}(\mathcal{O}_{Y,t}).
    \end{displaymath}
    Now, taking the supremum over all $t\in Y$, we have the desired inequality:
    \begin{displaymath}
        \begin{aligned}
        \sup_{s\in X} \left\{ \beta^{\mathcal{O}_{X,s}} (F_\ast^e \mathcal{O}_{X,s}) - \operatorname{depth}(\mathcal{O}_{X,s})  \right\} 
        &\geq \sup_{t\in Y} \left\{ \beta^{\mathcal{O}_{X,f(t)}} (F_\ast^e \mathcal{O}_{X,f(t)}) - \operatorname{depth}(\mathcal{O}_{X,f(t)})  \right\}
        \\&= \sup_{t\in Y} \left\{ \beta^{\mathcal{O}_{Y,t}} (F_\ast^e \mathcal{O}_{Y,t}) - \operatorname{depth}(\mathcal{O}_{Y,t}) \right\}.
        \end{aligned}
    \end{displaymath}
\end{proof}

\begin{notation}
    \label{not:gamme_function}
    Let $X$ be a Noetherian scheme and $E$ a coherent $\mathcal{O}_X$-module. Set $\gamma (E)$ to be the following,
    \begin{displaymath}
        \sup_{s\in X} \left\{ \beta^{\mathcal{O}_{X,s}} (E_s) - \operatorname{depth}(\mathcal{O}_{X,s})  \right\}.
    \end{displaymath}
\end{notation}

\begin{example}
    \label{ex:beta_depth_to_codepth}
    Let $X$ be an $F$-finite scheme. 
    Using \Cref{rmk:codepth_facts}, it can be checked that $\operatorname{codepth}(X)\leq \gamma (F^e_\ast \mathcal{O}_X)$ for every $e\geq 1$. 
    Indeed, this holds at the level of stalks, and so one can take the supremum over all $s\in X$. 
    Moreover, by arguing affine locally, $\gamma (F^e_\ast \mathcal{O}_X)<\infty$. 
    To see this, quasi-compactness of $X$ makes the problem affine local. Hence, we can assume $X$ is affine. 
    By \cite[Proposition 1.1]{Kunz:1976}, an $F$-finite ring has finite Krull dimension. 
    Hence, $0\leq \operatorname{depth}(\mathcal{O}_{X,p})\leq \dim X$ for all $p\in X$ \cite[\href{https://stacks.math.columbia.edu/tag/00LK}{Tag 00LK}]{stacks-project}. 
    As $X$ is $F$-finite and affine, there exists an $N_e\geq 0$ such that $\beta^{\mathcal{O}_X} (F^e_\ast \mathcal{O}_X)\leq N_e$. 
    Thus, for all $p\in X$, $\beta^{\mathcal{O}_{X,p}} (F^e_\ast \mathcal{O}_{X,p})\leq N_e$.
\end{example}

\section{Frobenius generation}
\label{sec:frob_gen}

We prove our results concerned with Frobenius generation.

\subsection{Schemes}
\label{sec:frob_gen_scheme}

\begin{lemma}
    \label{lem:generation_is_affine_local}
    Let $X$ be a Noetherian separated scheme. 
    Choose $E,G\in D^b_{\operatorname{coh}}(X)$ and $P\in \operatorname{Perf}(X)$ a compact generator for $D_{\operatorname{qc}}(X)$. 
    The following are equivalent:
    \begin{enumerate}
        \item \label{lem:generation_is_affine_local1} $E\in \langle P \otimes^{\mathbf{L}} G \rangle$
        \item \label{lem:generation_is_affine_local2} $\mathbf{L}j^\ast E \in \langle \mathbf{L}j^\ast G \rangle$ for all open immersions $j\colon U \to X$ from a quasi-affine scheme 
        \item \label{lem:generation_is_affine_local3} there exists an affine open cover $U_i$ of $X$ with associated open immersions $s_i \colon U_i \to X$ such that $\mathbf{L}s_i^\ast E \in \langle \mathbf{L}s_i^\ast G \rangle$.
    \end{enumerate}
\end{lemma}

\begin{proof}
    $\eqref{lem:generation_is_affine_local2} \implies \eqref{lem:generation_is_affine_local3}$ is immediate. 
    For $\eqref{lem:generation_is_affine_local1} \implies \eqref{lem:generation_is_affine_local2}$, let $s\colon U\to X$ be an open immersion. 
    Then $\mathbf{L}s^\ast \colon D^b_{\operatorname{coh}}(X) \to D^b_{\operatorname{coh}}(U)$ is a Verdier localization. 
    If $U$ is quasi-affine, then so $s$ by \cite[\href{https://stacks.math.columbia.edu/tag/054G}{Tag 054G}]{stacks-project}. 
    It follows from \cite[Lemma 2.6 and Example 3.11]{Hall/Rydh:2017} (see also \cite[Corollary 3.6]{Lank:2026}), $\mathbf{L}j^\ast P$ is a compact generator of $D_{\operatorname{qc}}(U)$. 
    Since $U$ is a quasi-affine scheme, $\operatorname{Perf}(U)=\langle\mathcal{O}_U\rangle$. 
    It follows that 
    \begin{displaymath}
        \mathbf{L}j^\ast E \in\mathbf{L}j^\ast\langle P \otimes^{\mathbf{L}} G \rangle\subseteq \langle \mathbf{L}j^\ast G\otimes^{\mathbf{L}}\mathbf{L}j^\ast P\rangle=\langle \mathbf{L}j^\ast G\otimes^{\mathbf{L}}\mathcal{O}_U\rangle=\langle \mathbf{L}j^\ast G \rangle
    \end{displaymath}
    as desired. 

    For $\eqref{lem:generation_is_affine_local3} \implies \eqref{lem:generation_is_affine_local1}$, we follow the cocovering argument of \cite[Theorem 5.15]{Rouquier:2008}. 
    Assume there exists an affine open cover $X=\cup^N_{i=1} U_i$ with associated open immersions $s_i \colon U_i \to X$ such that $\mathbf{L}s_i^\ast E \in \langle \mathbf{L}s_i^\ast G \rangle$. 
    For an arbitrary nonempty subset $I\subseteq\{1,\ldots,N\}$, set $U_I=\cap_{i\in I}U_i$ and let $s_I\colon U_I\to X$ be the canonical open immersion. Since $X$ is separated, $U_I$ is an affine scheme for every $I$. 
    We claim that
    \begin{displaymath}
        \mathbf{R}s_{I,\ast}\mathbf{L}s_I^\ast E\subseteq\overline{\langle P \otimes^{\mathbf{L}} G \rangle}_{m_I}
    \end{displaymath}
    for some $m_I>0$ for every $I$. 
    To see this, we first observe that we have $\mathbf{L}s_I^\ast E \in \langle \mathbf{L}s_I^\ast G \rangle$ for every $I$. 
    By \cite[Theorem 6.2]{Neeman:2021a}, there exists $n_I$ such that $\mathbf{R}s_\ast \mathcal{O}_{U_I} \in \overline{\langle  P \rangle}_{n_I}$. By the projection formula, we have
    \begin{displaymath}
        \mathbf{R}s_{I,\ast}\mathbf{L}s_I^\ast G\simeq G\otimes^{\mathbf{L}}\mathbf{R}s_{I,\ast}\mathcal{O}_U\subseteq G\otimes^{\mathbf{L}}\overline{\langle  P \rangle}_{n_I}\subseteq\overline{\langle G\otimes^{\mathbf{L}} P \rangle}_{n_I}.
    \end{displaymath}
    It follows that $\mathbf{R}s_{I,\ast}\mathbf{L}s_I^\ast E\in\overline{\langle G\otimes^{\mathbf{L}} P \rangle}_{m_I}$ for some $m_I$ as claimed. Observe that the subcategories $D_{\operatorname{qc},X\setminus U_i}(X)$ form a cocovering of $D_{\operatorname{qc}}(X)$. 
    By induction on $N$, we see that $E\in\overline{\langle G\otimes^{\mathbf{L}} P \rangle}_{N}$ for some $N>0$. 
    Indeed, we may assume, by induction, that $X=U\cup V$ where $U$ and $V$ are affine schemes. Let $s\colon U\to X$, $t\colon V\to X$ and $r\colon U\cap V\to X$. 
    By \cite[Proposition 5.10]{Rouquier:2008}, we have a Mayer--Vietoris triangle
    \begin{displaymath}
        E\longrightarrow\mathbf{R}s_{\ast}\mathbf{L}s^\ast E\oplus\mathbf{R}t_{\ast}\mathbf{L}t^\ast E\longrightarrow\mathbf{R}r_{\ast}\mathbf{L}r^\ast E\to E[1].
    \end{displaymath}
    By the induction hypothesis, we have $\mathbf{R}s_{\ast}\mathbf{L}s^\ast E\in\overline{\langle G\otimes^{\mathbf{L}} P \rangle}_{m_s}$, $\mathbf{R}t_{\ast}\mathbf{L}t^\ast E\in\overline{\langle G\otimes^{\mathbf{L}} P \rangle}_{m_t}$ and $\mathbf{R}r_{\ast}\mathbf{L}r^\ast E\in\overline{\langle G\otimes^{\mathbf{L}} P \rangle}_{m_r}$ for some $m_s,m_t,m_r>0$. Therefore we have $E\in\overline{\langle G\otimes^{\mathbf{L}} P \rangle}_{N}$ where $N=\max\{m_s,m_t\}+m_r+1$ as desired. 
    Since $E,G\in D_{\operatorname{coh}}^b(X)$, the result follows from \cite[Remark 2.19]{Lank:2024}.
\end{proof}

\begin{lemma}
    \label{lem:residue_field_frobenius_generation}
    Let $(R,\mathfrak{m},k)$ be an $F$-finite local ring of characteristic $p$. Denote by $K$ the Koszul complex on a minimal set of generators for $\mathfrak{m}$. Let $M\in D_{\operatorname{coh}}^b(R)$ such that $\mathfrak{m}\in\operatorname{supp}(M)$. 
    If $e \geq \lfloor \log_p (\gamma(F_\ast R))\rfloor+1$, then $k \in \langle F_\ast^e (M\otimes^\mathbf{L} K) \rangle_1$. In particular, $k \in \langle F^e_\ast (M\otimes^\mathbf{L}\operatorname{Perf}_{\{\mathfrak{m}\}} (R)) \rangle_1$. 
\end{lemma}

\begin{proof}
    By \Cref{rmk:codepth_facts}, $\operatorname{codepth}(R) \leq \gamma(F_\ast R)$. Since $e \geq \lfloor \log_p (\gamma(F_\ast R))\rfloor+1$, we have $p^e> \gamma(F_\ast R) \geq \operatorname{codepth}(R)$. Since $\mathfrak{m}\in\operatorname{supp}(M)$, we have $\mathfrak{m}\in\operatorname{supp}(M\otimes^{\mathbf{L}}K)$. 
    Let $A$ be the animated Koszul $R$-algebra whose underlying complex is $K$.
    Then the proof of \cite[Theorem 4.1.13]{Bhatt/Blickle/Schwede/Tucker:2026} shows the $e$-fold Frobenius map factors over $A\to \pi_0(A)=k$. It follows that $F_\ast^e(M\otimes^{\mathbf{L}}K)$ is a nonzero complex of finite dimensional $k$-vector spaces. Therefore we have $k \in \langle F_\ast^e(M\otimes^{\mathbf{L}}K \rangle_1$ as desired. 
    The last claim follows from the fact $F_\ast^e$ is an endofunctor on $D^b_{\operatorname{coh},\{\mathfrak{m}\}} (R)$.
\end{proof}

\begin{proposition}
    \label{prop:relative_Frobenius_generation_local_ring}
    Let $X=\operatorname{Spec}(R)$ where $R$ is an $F$-finite ring. 
    Consider a closed subset $Z\subseteq X$ and $e \geq \lfloor \log_p (\gamma(F_\ast \mathcal{O}_X))\rfloor + 1$.
    Let $G$ be a compact generator of $D_{\operatorname{qc},Z}(X)$. For every $M\in D_{coh}^b(X)$ with $Z\subseteq\operatorname{supp}(M)$, we have
    \begin{displaymath}
        D^b_{\operatorname{coh},Z}(X)= \langle F_\ast^e (M\otimes^{\mathbf{L}}G) \rangle.
    \end{displaymath}
\end{proposition}

\begin{proof}
    We first assume that $R$ is a local ring. 
    If $Z$ is empty, there is nothing to prove.
    Suppose $Z$ is nonempty.
    We proceed by induction on the Krull dimension of $R$. 
    Denote by $i\colon \operatorname{Spec}(\kappa(q))\to X$ the closed immersion of the unique closed point $q\in X$. 
    Set $j\colon U \to X$ for the open immersion associated to $U:=X\setminus \{q\}$. 
    If $\dim X = 0$, then $X$ is Artinian. 
    Let $K$ be the Koszul complex on a minimal set of generators for $\mathfrak{m}_q$
    By \Cref{lem:residue_field_frobenius_generation}, we have $i_\ast \mathcal{O}_{\operatorname{Spec}(\kappa(q))} \in \langle F_\ast^e (M\otimes^\mathbf{L} K) \rangle$.
    But $K$ is compact and thus $K\in\langle G\rangle$. It follows that $i_\ast \mathcal{O}_{\operatorname{Spec}(\kappa(q))} \in \langle F_\ast^e (M\otimes^\mathbf{L} K)\rangle\subseteq\langle F_\ast^e (M\otimes^\mathbf{L} G)\rangle$, which proves the base case. 

    Now, assume we have proved the claim for $F$-finite local rings of Krull dimension at most $0 \leq N$. 
    Suppose $\dim X = N +1$. 
    By \Cref{lem:relative_denseness_open_immersion}, there exists a Verdier localization sequence
    \begin{displaymath}
        D^b_{\operatorname{coh},\{q\}} (X) \to D^b_{\operatorname{coh},Z}(X) \xrightarrow{\mathbf{L}j^\ast} D^b_{\operatorname{coh},U\cap Z}(U).
    \end{displaymath}
    Observe, from \Cref{lem:residue_field_frobenius_generation}, $i_\ast \mathcal{O}_{\operatorname{Spec}(\kappa(q))} \in \langle F^e_\ast (M\otimes^\mathbf{L}\operatorname{Perf}_{\{q\}} (X)) \rangle\subseteq\langle F_\ast^e (M\otimes^\mathbf{L}G)\rangle$. 
    Since $X$ has a unique closed point $q$, it follows that $\dim U \leq N$. 
    Choose an affine open cover $U_i$ for $U$ with associated open immersions $s_i \colon U_i \to U$. 
    By \Cref{rm:etale_base_change_Frobenii}, $\mathbf{L}s_i^\ast F_\ast^e G \cong F_\ast^e \mathbf{L}s_i^\ast G$. 
    Moreover, \Cref{rmk:relative_perf_dense_localization} says $\langle \mathbf{L}s_i^\ast G \rangle = \operatorname{Perf}_{Z \cap U_i} (U_i)$. 
    Moreover, we have $\gamma(F_*\mathcal{O}_{X,u})\leq\gamma(F_*\mathcal{O}_X)$. Let $E\in D_{\operatorname{coh}^{b},Z\cap U_i}(U_i)$. By induction hypothesis, we have $E_u\in\langle F_\ast^e (M_u\otimes^\mathbf{L} G_u)\rangle$ for every $u\in U_i$.
    Applying \cite[Theorem 3.6]{Letz:2021}, we obtain $\mathbf{L}s_i^\ast E\in \langle F_\ast^e \mathbf{L}s_i^\ast (M\otimes^\mathbf{L}G) \rangle$. 
    Since this is true for every $i$, we see that $\mathbf{L}s^\ast F_\ast^e (M\otimes^\mathbf{L}G)$ is a classical generator for $D^b_{\operatorname{coh},Z\cap U}(U)$ by \Cref{lem:generation_is_affine_local}. because $\mathcal{O}_U$ is a compact generator for $D_{\operatorname{qc}}(U)$. 
    Thus, $F_\ast^e (M\otimes^\mathbf{L}G)$ is a classical generator for $D^b_{\operatorname{coh},Z}(X)$ (see e.g.\ \cite[Proposition 4.3]{Elagin/Lunts/Schnurer:2020}). 
    This finishes the proof.

    For a general $R$, let $E\in D_{\operatorname{coh},Z}^b(R)$. Let $\mathfrak{p}\subseteq R$ be a prime ideal. The local case tells us that $E_\mathfrak{p}\in \langle F_\ast^e (M_\mathfrak{p}\otimes^\mathbf{L}G_\mathfrak{p})\rangle$. Since $\mathfrak{p}$ is arbitrary, we may apply \cite[Theorem 3.6]{Letz:2021} and the result follows.
\end{proof}

\begin{corollary}
    \label{prop:relative_Frobenius_generation_schemes}
    Let $X$ be a Noetherian separated $F$-finite scheme. 
    Let $Z\subseteq X$ be a closed subset. 
    Let $e \geq \lfloor \log_p (\gamma(F_\ast \mathcal{O}_X))\rfloor + 1$. Let $P$ be a compact generator of $D_{\operatorname{qc},Z}(X)$. 
    For every object $G\in D_{\operatorname{coh}}^b(X)$ with $\operatorname{supp}(G)=Z$, we have
    \begin{displaymath}
        D^b_{\operatorname{coh},Z}(X)= \langle F_\ast^e (G\otimes^\mathbf{L}P) \rangle.
    \end{displaymath}
\end{corollary}

\begin{proof}
    Set $H=F_\ast^e (G\otimes^\mathbf{L}P)$. Let $E\in D^b_{\operatorname{coh},Z}(X)$. Choose an affine open covering $U_i$ of $X$ with associated open immersions $s_i\colon U_i\to U$. By \Cref{prop:relative_Frobenius_generation_local_ring}, we have $\mathbf{L}s_i^\ast E\in\langle \mathbf{L}s_i^\ast H\rangle$ for every $i$. Let $Q$ be a compact generator of $D_{\operatorname{qc}}(X)$. By \Cref{lem:generation_is_affine_local}, we see that $E\in\langle Q\otimes^\mathbf{L} H\rangle$. By the projection formula, we have$$Q\otimes^\mathbf{L} H\simeq F_\ast^e(F^{\ast,e}Q\otimes^\mathbf{L}G\otimes^\mathbf{L}P)\simeq F_\ast^e(G\otimes^\mathbf{L}F^{\ast,e}Q\otimes^\mathbf{L}P).$$But $F^{\ast,e}Q\otimes^\mathbf{L}P$ is a perfect complex supported on $Z$. It follows that $F_{\ast,e}Q\otimes^\mathbf{L}P\in\langle P\rangle$ in $D_{\operatorname{qc},Z}(X)$. Therefore we have that$$E\in\langle Q\otimes^\mathbf{L} H\rangle\subseteq \langle F_\ast^e (G\otimes^\mathbf{L}P) \rangle$$and the result follows.
\end{proof}

\begin{remark}
    \hfill
    \begin{enumerate}
        \item As a consequence of \Cref{prop:relative_Frobenius_generation_local_ring} and \Cref{lem:generation_is_affine_local}, \Cref{prop:relative_Frobenius_generation_schemes} recovers the main result of \cite{Ballard/Iyengar/Lank/Mukhopadhyay/Pollitz:2023}. In fact, this occurs by independent methods. See \Cref{sec:intro_what_we_do_frob_gen} for discussion.
        \item  Note that $\langle F_\ast^e (G\otimes^\mathbf{L}P) \rangle$ would rarely be a strong generator of $D^b_{\operatorname{coh},Z}(X)$. 
        See \cite[Remark 8.1]{Neeman:2022}.
    \end{enumerate}
\end{remark}

\subsection{Stacks}
\label{sec:frob_gen_stacks}

\begin{hypothesis}
    \label{hyp:frob_gen}
    Let $\mathcal{X}$ be a Noetherian algebraic stack. 
    We say that $\mathcal{X}$ satisfies \Cref{hyp:frob_gen} if for every $Z\subseteq |\mathcal{X}|$ closed there exists $e_0 \geq 0$ such that $D^b_{\operatorname{coh},Z}(\mathcal{X})= \langle \mathbf{R} F_\ast^e \operatorname{Perf}_Z (\mathcal{X}) \rangle$ for every $e\geq e_0$. 
\end{hypothesis}

\begin{proposition}
    \label{prop:open_immersion}
    Let $\mathcal{X}$ be a Noetherian $F$-finite algebraic stack. 
    Consider an open immersion $j\colon \mathcal{U} \to \mathcal{X}$. 
    If $\mathcal{X}$ satisfies \Cref{hyp:frob_gen}, then so does $\mathcal{U}$.
\end{proposition}

\begin{proof}
    By \Cref{lem:F-finiteness_ascent_descent}, we know that $\mathcal{U}$ is $F$-finite. 
    Choose a closed subset $Z\subseteq |\mathcal{U}|$. 
    Define $W= \overline{Z}$ (i.e.\ the closure of $Z$ in $|\mathcal{X}|$). 
    By \Cref{lem:relative_denseness_open_immersion}, $\mathbf{L} j^\ast$ restricts to a Verdier localization $D^b_{\operatorname{coh},W}(\mathcal{X}) \to D^b_{\operatorname{coh},Z}(\mathcal{U})$. 
    As $\mathcal{X}$ satisfies \Cref{hyp:frob_gen}, it follows that $\langle \mathbf{R} F^e_\ast \operatorname{Perf}_W (\mathcal{X}) \rangle = D^b_{\operatorname{coh},W}(\mathcal{X})$ for $e\gg 0$, and so, $\langle \mathbf{L} j^\ast  \mathbf{R} F^e_\ast \operatorname{Perf}_W (\mathcal{X}) \rangle = D^b_{\operatorname{coh},Z}(\mathcal{U})$. 
    From $j$ being an open immersion, it is \'{e}tale and so \Cref{rm:etale_base_change_Frobenii} tells us $\mathbf{L} j^\ast \mathbf{R} F^e_\ast E \cong \mathbf{R} F_\ast^e \mathbf{L} j^\ast E$ for all $E\in D_{\operatorname{qc}}(\mathcal{X})$. 
    Hence, $\langle \mathbf{R} F^e_\ast \mathbf{L} j^\ast \operatorname{Perf}_W (\mathcal{X}) \rangle = D^b_{\operatorname{coh},Z}(\mathcal{U})$, which shows that $\mathcal{U}$ satisfies \Cref{hyp:frob_gen} because $\mathbf{L} j^\ast \operatorname{Perf}_W (\mathcal{X}) \subseteq \operatorname{Perf}_Z (\mathcal{U})$.
\end{proof}

\begin{lemma}
    \label{lem:finite_cover_quotients_for_support}
    Let $\mathcal{X}$ be a concentrated Noetherian algebraic stack satisfying approximation by compacts. 
    Suppose there is a finite flat surjective morphism $f\colon V \to \mathcal{X}$ from an affine scheme. 
    For any closed subset $Z\subseteq |\mathcal{X}|$, the functor $\mathbf{R}f_\ast \colon D^b_{\operatorname{coh},f^{-1}(Z)}(V) \to D^b_{\operatorname{coh},Z}(\mathcal{X})$ is essentially dense (i.e.\ everything in the target category is a direct summand of something in the essential image).
\end{lemma}

\begin{proof}
    By \cite[Lemma 6.2]{DeDeyn/Lank/ManaliRahul/Peng:2025}, the unit morphism of derived pullback/pushforward $E \to \mathbf{R}f_\ast\mathbf{L}f^\ast E$ splits for all $E\in D_{\operatorname{qc}}(\mathcal{X})$. 
    Choose $E\in D^b_{\operatorname{coh},Z}(\mathcal{X})$. 
    Since $f$ is finite and faithfully flat, we have $\mathbf{L}f^\ast E \in D^b_{\operatorname{coh},f^{-1}(Z)}(V)$. 
    Yet, flatness of $f$ implies $\mathbf{L}f^\ast E$ has bounded cohomology. Thus, the claim follows.
\end{proof}

\begin{remark}
    \label{rm:finite_cover_quotients_for_support}
    \Cref{lem:finite_cover_quotients_for_support} is false if $\mathcal{X}$ is not concentrated. 
    For example, let $V=\operatorname{Spec}(\mathbb{F}_p)$, $\mathcal{X}=B\mathbb{Z}/p\mathbb{Z}$ and $f\colon V\to\mathcal{X}$ be the canonical covering. 
    Then the natural morphism $\mathcal{O_X}\to f_\ast\mathcal{O}_V$ does not split. 
    Indeed, we see that $f_\ast\mathcal{O}_V$ is the regular representation of $\mathbb{Z}/p\mathbb{Z}$ over $\mathbb{F}_p$ and it does not contain $\mathcal{O}_{\mathcal{X}}$, which is the trivial representation of $\mathbb{Z}/p\mathbb{Z}$.
\end{remark}

\begin{proposition}
    \label{prop:finite_cover}
    Let $\mathcal{X}$ be a concentrated $F$-finite algebraic stack satisfying approximation by compacts.
    If there is a finite flat surjective morphism $f\colon V \to \mathcal{X}$ from an affine scheme, then $\mathcal{X}$ satisfies \Cref{hyp:frob_gen}.
\end{proposition}

\begin{proof}
    Note that $V$ is $F$-finite by \Cref{lem:F-finiteness_ascent_descent}. Let $Z\subseteq |\mathcal{X}|$ be a closed subset. 
    By \Cref{prop:relative_Frobenius_generation_local_ring}, $V$ satisfies \Cref{hyp:frob_gen}. 
    As $\mathcal{X}$ is concentrated and satisfies approximation by compacts, \Cref{lem:finite_cover_quotients_for_support} implies $\mathbf{R} f_\ast \colon D^b_{\operatorname{coh},f^{-1}(Z)}(V) \to D^b_{\operatorname{coh},Z}(\mathcal{X})$ is essentially dense. 
    We can find an $e>0$ such that 
    \begin{displaymath}
        \langle F_\ast^e \operatorname{Perf}_{f^{-1} (Z)} (V) \rangle = D^b_{\operatorname{coh},f^{-1}(Z)}(V).
    \end{displaymath}
    Hence, we have that $\langle \mathbf{R} f_\ast F_\ast^e \operatorname{Perf}_{f^{-1} (Z)} \rangle = D_{\operatorname{coh},Z}^b(\mathcal{X})$. 
    However, there is a natural isomorphism of functors $\mathbf{R} f_\ast F_\ast^e \to \mathbf{R} F_\ast^e \mathbf{R} f_\ast$ and $\mathbf{R}f_\ast\operatorname{Perf}_{f^{-1} (Z)} (V) \subseteq \operatorname{Perf}_Z(\mathcal{X})$ since $f$ is finite and flat.
    It follows that $\mathcal{X}$ satisfies \Cref{hyp:frob_gen}. 
\end{proof}

\begin{remark}
    Recall an \textbf{\'{e}tale neighborhood} is an open immersion $i\colon \mathcal{U} \to \mathcal{X}$ and \'{e}tale morphism $\mathcal{Y} \xrightarrow{f} \mathcal{X}$ which is an isomorphism over $|\mathcal{X}|\setminus |\mathcal{U}|$ (endowed with reduced induced substack structure).
\end{remark}

\begin{proposition}
    \label{prop:etale_nbhd}
    Let $\mathcal{X}$ be a concentrated $F$-finite algebraic stack. 
    Consider an \'{e}tale neighborhood
    \begin{displaymath}
        \begin{tikzcd}
            f^{-1}(\mathcal{U})\arrow[r,"i"]\arrow[d,"g"] & \mathcal{Y}\arrow[d,"f"]\\
            \mathcal{U}\arrow[r,"j"] & \mathcal{X}
        \end{tikzcd}
    \end{displaymath}
    where $j$ is an open immersion and $f$ is a quasi-compact morphism which is representable by algebraic spaces. 
    Suppose both $\mathcal{X}$ and $\mathcal{Y}$ satisfy the Thomason condition. 
    If $\mathcal{U}$ and $\mathcal{Y}$ satisfies \Cref{hyp:frob_gen}, then so does $\mathcal{X}$. 
\end{proposition}

\begin{proof}
    By \Cref{lem:F-finiteness_ascent_descent}, both $\mathcal{U}$ and $\mathcal{Y}$ are $F$-finite. Choose a closed subset $W\subseteq |\mathcal{X}|$. 
    Set $Z:=|\mathcal{X}|\setminus |\mathcal{U}|$. 
    Using \Cref{lem:relative_denseness_open_immersion}, there is a Verdier localization
    \begin{displaymath}
        D^b_{\operatorname{coh}Z\cap W}(\mathcal{X}) \to D^b_{\operatorname{coh},W}(\mathcal{X}) \to D^b_{\operatorname{coh},W\cap |\mathcal{U}|}(\mathcal{U}).
    \end{displaymath}
    Also, there is another Verdier localization
    \begin{displaymath} 
       D_{\operatorname{qc},Z\cap W}(\mathcal{X}) 
       \xrightarrow{i_\ast} D_{\operatorname{qc},W}(\mathcal{X}) 
       \xrightarrow{\mathbf{L} j^\ast} D_{\operatorname{qc},W\cap |\mathcal{U}|}(\mathcal{U})
    \end{displaymath}
    where $i_\ast$ is the natural inclusion (see e.g.\ the proof of \cite[Lemma 5.9]{DeDeyn/Lank/ManaliRahul/Peng:2025}). 
    From \cite[Theorem 2.1]{Neeman:1996}, it follows $\mathbf{L} j^\ast\colon \operatorname{Perf}_W (\mathcal{X}) \to \operatorname{Perf}_{W\cap |\mathcal{U}|}(\mathcal{U})$ is a Verdier localization up to direct summands.
    By \cite[Theorem 4.2]{Hall/Rydh:2023}, there is an equivalence $D_{\operatorname{qc},f^{-1}(Z)}(\mathcal{Y}) \to D_{\operatorname{qc},Z}(\mathcal{X})$ induced by the derived pushforward/pullback adjunction along $f$.
    It follow from the hypothesis that $D^b_{\operatorname{coh},Z\cap W}(\mathcal{X}) = \langle  \mathbf{R} F_\ast^e \operatorname{Perf}_{Z\cap W} (\mathcal{X}) \rangle$ if $e\gg 0$. 
    Moreover, from our hypothesis on $\mathcal{U}$, we see that if $e\gg 0$, then
    \begin{displaymath}
        \langle \mathbf{R} F_\ast^e \mathbf{L} j^\ast \operatorname{Perf}_W (\mathcal{X}) \rangle = D^b_{\operatorname{coh},W\cap |\mathcal{U}|}(\mathcal{U}).
    \end{displaymath}
    However, $j$ is \'{e}tale, so $\mathbf{R} F_\ast^e \mathbf{L} j^\ast E \cong \mathbf{L} j^\ast \mathbf{R} F_\ast^e E$ for each $E\in \operatorname{Perf}_W (\mathcal{X})$ via \Cref{rm:etale_base_change_Frobenii}. 
    Hence, $\langle \mathbf{L} j^\ast \mathbf{R} F_\ast^e \operatorname{Perf}_W (\mathcal{X}) \rangle = D^b_{\operatorname{coh},W\cap |\mathcal{U}|}(\mathcal{U})$. 
    It follows that
    \begin{displaymath}
        D^b_{\operatorname{coh},W}(\mathcal{X}) \subseteq \langle \mathbf{R} F_\ast^e \operatorname{Perf}_{Z\cap W} (\mathcal{X}) \oplus \mathbf{R} F_\ast^e \operatorname{Perf}_W (\mathcal{X}) \rangle = \langle \mathbf{R} F_\ast^e \operatorname{Perf}_W (\mathcal{X}) \rangle,
    \end{displaymath}
    see e.g.\ \cite[Lemma's 5.5 \& 5.6]{DeDeyn/Lank/ManaliRahul:2025}. 
    Thus, the claim follows.
\end{proof}

\begin{theorem}
    \label{thm:Frobenius_generation}
    Let $\mathcal{S}$ be a Noetherian $F$-finite algebraic stack with quasi-finite and separated diagonal. 
    If $\mathcal{S}$ is concentrated, then $\mathcal{S}$ satisfies \Cref{hyp:frob_gen}. 
    In fact, for any $Z\subseteq |\mathcal{S}|$ closed and $e \gg 0$, there is a compact generator $G$ of $D_{\operatorname{qc},Z} (\mathcal{S})$ such that $D^b_{\operatorname{coh},Z}(\mathcal{S})= \langle \mathbf{R} F_\ast^e G \rangle$.
\end{theorem}

\begin{proof}
    Define $\mathbb{E}$ to be the strictly full $2$-subcategory of algebraic stacks over $\mathcal{S}$ consisting of algebraic stacks whose structure morphism $\mathcal{X} \to \mathcal{S}$ are representable by algebraic spaces, 
    separated, finitely presented, quasi-finite, and flat. 
    We make a few observations regarding the objects of $\mathbb{E}$:
    \begin{itemize}
        \item The source of object in $\mathbb{E}$ is concentrated. 
        Indeed, each morphism in $\mathbb{E}$ is representable by algebraic spaces (see e.g.\ \cite[Lemma 6.7]{DeDeyn/Lank/ManaliRahul:2025}), and hence, must be concentrated via \cite[Lemma 2.5(3)]{Hall/Rydh:2017}. 
        As $\mathcal{S}$ is concentrated, the source of objects in $\mathbb{E}$ is concentrated.
        \item The source of object in $\mathbb{E}$ is $F$-finite. 
        This follows from the fact that every morphism of $\mathbb{E}$ is representable by algebraic spaces and of finite presentation. 
        Hence, \Cref{lem:F-finiteness_ascent_descent} is applicable.
        \item The source of object in $(f\colon \mathcal{X}\to \mathcal{S}) \in \mathbb{E}$ has separated and quasi-finite diagonal. 
        Indeed, each morphism is representable by algebraic spaces. 
        The diagonal of $\mathcal{X}$ is given by the composition $\mathcal{X} \xrightarrow{\Delta_f} \mathcal{X} \times_{\mathcal{S}}\mathcal{X} \to \mathcal{X} \times_{\mathbb{Z}}\mathcal{X}$ where the right morphism is obtained from \cite[\href{https://stacks.math.columbia.edu/tag/04Z1x}{Tag 04Z1}]{stacks-project}. 
        Since $f$ is representable by algebraic spaces, the morphism $\Delta_f$ is a monomorphism and the composition is quasi-finite and separated as desired.
        \item The source of the object in $\mathbb{E}$ satisfies approximation by compacts and the $1$-Thomason condition. 
        Indeed, these respectively follow from \cite[Lemma 4.6(i)]{Hall/Rydh:2019}, \cite[Theorem A]{Hall/Lamarche/Lank/Peng:2025}, and \cite[Theorem A]{Hall/Rydh:2017} coupled with the observations on diagonals.
    \end{itemize}
    Set $\mathbb{D}$ to be the subcategory of objects in $\mathbb{E}$ whose sources satisfy \Cref{hyp:frob_gen}. 
    We use \cite[Theorem E]{Hall/Rydh:2018} to show\footnote{We remind the reader (I2) of \cite[Theorem E]{Hall/Rydh:2018} requires \textit{flatness}.} that $\mathbb{D}=\mathbb{E}$. 
    This requires verifying the following:
    \begin{itemize}
        \item If $(\mathcal{U} \to \mathcal{X})\in\mathbb{E}$ is an open immersion and $\mathcal{X}\in\mathbb{D}$, then $\mathcal{U}\in\mathbb{D}$.
        \item If $(V \to \mathcal{X})\in \mathbb{E}$ is finite, flat and surjective with affine source, then $\mathcal{X}\in\mathbb{D}$.
        \item If $(\mathcal{U} \xrightarrow{i} \mathcal{X})$, $(\mathcal{Y} \xrightarrow{f} \mathcal{X})\in\mathbb{E}$ form an \'{e}tale neighborhood, then $\mathcal{X}\in \mathbb{D}$ whenever $\mathcal{U}$, $\mathcal{Y}\in\mathbb{D}$.
    \end{itemize}
    However, these are exactly \Cref{prop:open_immersion,prop:finite_cover,prop:etale_nbhd}. 
    Then, we see that $\mathcal{S}$ satisfies \Cref{hyp:frob_gen}. 
    Moreover, if coupled with the fact that the source of every object in $\mathbb{E}$ satisfies the $1$-Thomason condition, the last claim follows.
\end{proof}

\begin{example}
    \label{ex:non_F-finite_false}
    Let $k=\mathbb{F}_2$.
    Let $X=\operatorname{Spec}(k[x,y]/(xy))$. 
    Consider the stack $\mathcal{X}=X\times_k B\alpha_2$ over $k$. 
    We know from \Cref{lem:F-finite_classifying_stack} that $B\alpha_2$ is not $F$-finite. 
    We claim that \Cref{hyp:frob_gen} fails for $\mathcal{X}$ and $Z=|\mathcal{X}|$ for every $e\geq 0$. 
    We first see that $\mathbf{R}\pi_\ast\mathcal{O}_X$ is unbounded on $\mathcal{X}$ for every $e>0$. 
    This follows from flat base change and the fact that $\mathbf{R}\Gamma(B\alpha_2, \mathcal{O}_{B\alpha_2})$ is unbounded as $\alpha_2$ is not concentrated \cite[Remark 4.6]{Hall/Rydh:2017}. 
    For $e=0$, we notice $\mathcal{X}$ is singular since $X$ is so. Therefore we have $\operatorname{Perf}(\mathcal{X})\neq D_{\operatorname{coh}}^b(\mathcal{X})$ and the claim follows.
\end{example}

\begin{remark}\label{ex:F-finite_non_tame_maybe}
    From \Cref{ex:non_F-finite_false}, we see that the Frobenius pushforward of bounded complexes may not be bounded and the obstructions lie in the high direct images of the Frobenius morphism. 
    These obstructions do not exist on $F$-finite stacks. 
    In particular, consider the stack $\mathcal{Y}=X\times_k B(\mathbb{Z}/p\mathbb{Z})$ in the notation of \Cref{ex:non_F-finite_false}. 
    We claim that $\mathcal{Y}$ is $F$-finite but not concentrated. 
    It is clear that $\mathcal{Y}$ is not concentrated due to the presence of $\mathbb{Z}/p\mathbb{Z}$ stabilizers. 
    Moreover, it is $F$-finite because it is a Deligne--Mumford stack locally of finite type over $\mathbb{F}_p$ by \Cref{ex:DM_F-finite}. 
    In particular, the Frobenius morphism $F_\mathcal{Y}$ is indeed finite as it is a separated representable universal homeomorphism that is locally of finite type by \Cref{lem:homeomorphism_on_underlying_topological_space} and  \Cref{cor:Frobenius_for_DM_is_representable}. Therefore $\mathbf{R}F_{\mathcal{Y},\ast}$ preserves boundedness. 
    It would be interesting to determine if some variants of \Cref{hyp:frob_gen} would hold for stacks like $\mathcal{Y}$.
\end{remark}

\begin{proof}
    [Proof of \Cref{cor:kunz}]
    First, let $Z$ be contained in the regular locus. 
    By \cite[Proposition 3.6]{DeDeyn/Lank/ManaliRahul/Peng:2025}, $D^b_{\operatorname{coh},Z}(\mathcal{X})=\operatorname{Perf}_Z (\mathcal{X})$.  

    Next, suppose $\mathbf{R} F_\ast^e G\in \operatorname{Perf}_Z (\mathcal{X})$ for all $e \gg 0$ and $G$ an arbitrary classical generator for $\operatorname{Perf}_Z (\mathcal{X})$. 
    By \Cref{thm:Frobenius_generation}, $D^b_{\operatorname{coh},Z}(\mathcal{X}) = \langle \mathbf{R} F_\ast^e G \rangle$ for all $e\gg 0$. 
    Hence, $D^b_{\operatorname{coh},Z}(\mathcal{X})=\operatorname{Perf}_Z (\mathcal{X})$. 
    Thus, from \cite[Proposition 3.6]{DeDeyn/Lank/ManaliRahul/Peng:2025}, $Z$ is contained in the regular locus.
\end{proof}

\begin{corollary}
    \label{cor:strong_frobenius_generator}
    Let $\mathcal{S}$ be a Noetherian separated concentrated $F$-finite Deligne--Mumford stack. 
    If $G$ is a classical generator for $\operatorname{Perf}(\mathcal{S})$ and $e\gg 0$, then $\mathbf{R}F_\ast^e G$ is a strong generator for $D^b_{\operatorname{coh}}(\mathcal{S})$. 
\end{corollary}

\begin{proof}
    Let $U \to \mathcal{S}$ be an \'{e}tale surjective morphism from an affine scheme. 
    Note that $U \to \mathcal{S}$ must be separated. 
    From \Cref{prop:F-finite_via_finite_presentation}, $U$ must be $F$-finite. By \cite[Proposition 1.1]{Kunz:1976}, $U$ has finite Krull dimension. Moreover, Theorem 2.5 of loc.\ cit.\ ensures $U$ must be excellent. Hence, \cite[Main Theorem]{Aoki:2021} implies $D^b_{\operatorname{coh}}(U)$ admits a strong generator. 
    Then \cite[Corollary 6.10]{DeDeyn/Lank/ManaliRahul:2025} implies $D^b_{\operatorname{coh}}(\mathcal{X})$ must as well. 
    Thus, the claim follows from \Cref{thm:Frobenius_generation} and \cite[\href{https://stacks.math.columbia.edu/tag/0FXA}{Tag 0FXA}]{stacks-project}.
\end{proof}

\begin{proposition}
    \label{prop:Frobenius_generation_bound}
    Let $\mathcal{S}$ be a Noetherian concentrated $F$-finite Deligne--Mumford stack with separated diagonal. 
    Then $e$ in \Cref{hyp:frob_gen} can be taken to be any integer at least
    \begin{displaymath}
        c:=\lfloor \log_p (\gamma (F_\ast \mathcal{O}_U )) \rfloor + 1.
    \end{displaymath}
    for some \'etale presentation $s\colon U\to \mathcal{S}$ from an affine scheme. 
    In particular, $c$ is always finite.
\end{proposition}

\begin{proof}
    The last claim for finiteness of $c$ follows from \Cref{ex:beta_depth_to_codepth}. 
    As $\mathcal{S}$ is Noetherian with quasi-finite and separated diagonal, it always admits an \'etale presentation $s\colon U\to \mathcal{S}$ from an affine scheme \cite[Theorem 7.2]{Rydh:2011}. 
    Moreover, \Cref{thm:Frobenius_generation} ensures it satisfies \Cref{hyp:frob_gen} as our hypothesis impose it has separated diagonal.

    Define $\mathbb{E}$ to be the strictly full $2$-subcategory of algebraic stacks over $\mathcal{S}$ consists of algebraic stacks whose structure morphisms $\mathcal{X}\to \mathcal{S}$ are representable by algebraic spaces, separated, finitely presented, and \'{e}tale. 
    We make a few observations about objects of $\mathbb{E}$:
    \begin{itemize}
        \item The source of object in $\mathbb{E}$ is Deligne--Mumford. 
        To see, note that each morphism is representable by algebraic spaces, and hence, Deligne--Mumford. 
        It follows from the fact $\mathcal{S}$ is Deligne--Mumford.
        \item The source of object in $\mathbb{E}$ is concentrated, affine-pointed, $F$-finite, satisfies the $1$-Thomason condition, has separated diagonals, and the property of approximation by compact. 
        These follow from similar reasoning as in \Cref{thm:Frobenius_generation}.
        \item Every morphism in $\mathbb{E}$ is \'{e}tale (see e.g.\ \cite[\href{https://stacks.math.columbia.edu/tag/0CIR}{Tag 0CIR}]{stacks-project}).
    \end{itemize}
    Set $\mathbb{D}$ to be the subcategory of objects in $\mathbb{E}$ whose sources satisfy the desired claim. 
    In particular, $\mathbb{D}$ consists of those objects $(g\colon\mathcal{X}\to \mathcal{S})\in \mathbb{E}$ whose source satisfy the property for any $Z\subseteq |\mathcal{S}|$ closed and $e\geq c$, there is a $G\in \operatorname{Perf}_{g^{-1}(Z)} (\mathcal{X})$ such that $D^b_{\operatorname{coh},g^{-1}(Z)}(\mathcal{X})= \langle \mathbf{R} F_\ast^eG \rangle$.
    
    We use \cite[Theorem E]{Hall/Rydh:2018} to show that $\mathbb{D}=\mathbb{E}$. 
    This requires verifying the following:
    \begin{enumerate}
        \item \label{prop:Frobenius_generation_bound1} If $(\mathcal{U} \to \mathcal{X})\in\mathbb{E}$ is an open immersion and $\mathcal{X}\in\mathbb{D}$, then $\mathcal{U}\in\mathbb{D}$.
        \item \label{prop:Frobenius_generation_bound2} If $(V \to \mathcal{X})\in \mathbb{E}$ is finite, flat and surjective with affine source, then $\mathcal{X}\in\mathbb{D}$.
        \item \label{prop:Frobenius_generation_bound3} If $(\mathcal{U} \xrightarrow{i} \mathcal{X})$, $(\mathcal{Y} \xrightarrow{f} \mathcal{X})\in\mathbb{E}$, where $i$ is an open immersion, $f$ is \'{e}tale and an isomorphism over $|\mathcal{X}|\setminus |\mathcal{U}|$ (endowed with reduced induced substack structure), i.e.\ an \'{e}tale neighborhood, then $\mathcal{X}\in \mathbb{D}$ whenever $\mathcal{U}$, $\mathcal{Y}\in\mathbb{D}$.
    \end{enumerate}
    We can argue in a similar fashion in \Cref{thm:Frobenius_generation} to show \eqref{prop:Frobenius_generation_bound1} and \eqref{prop:Frobenius_generation_bound3} are satisfied. 
    Indeed, one can observe that the arguments in these steps do not depend on the choice of $e\geq c$. 
    We are left to show \eqref{prop:Frobenius_generation_bound2}. 

    Suppose $(V \to \mathcal{X})\in \mathbb{E}$ is finite, flat, and surjective with an affine source. 
    From the argument of \eqref{prop:Frobenius_generation_bound2} in \Cref{thm:Frobenius_generation}, it suffices to show that the structure morphism satisfies $V\to \mathcal{S}\in \mathbb{D}$. 
    Consider the fibered square
    \begin{displaymath}
        \begin{tikzcd}
            {V\times_\mathcal{S} U} & U \\
            V & {\mathcal{S}.}
            \arrow["{h^\prime}", from=1-1, to=1-2]
            \arrow["{s^\prime}"', from=1-1, to=2-1]
            \arrow["s", from=1-2, to=2-2]
            \arrow["h"', from=2-1, to=2-2]
        \end{tikzcd}
    \end{displaymath}
    Note that $s$ is separated as it is a morphism from a separated scheme (i.e.\ $U$ is affine) to an algebraic stack with separated diagonal. 
    By base change, we know that $s^\prime$ is a separated, \'{e}tale, representable and thus quasi-affine by \cite[Proposition 3.1]{Rydh:2015}. 
    As $h$ is representable by algebraic spaces, the fiber product is an algebraic space. 
    Therefore the fiber product $V\times_\mathcal{S} U$ must be a quasi-affine scheme. 
    By base change, $s^\prime$ is an \'{e}tale surjective morphism. 
    Applying \Cref{lem:bound_along_etale} to $h^\prime$, we see that
    \begin{displaymath}
        c \geq \lfloor \log_p ( \gamma (F_\ast \mathcal{O}_{V\times_\mathcal{S} U}) ) \rfloor +1.
    \end{displaymath}

    Now, let $Z\subseteq |\mathcal{X}|$ be closed. 
    As stated above, it suffices to show that $D^b_{\operatorname{coh},f^{-1}(Z)}(V) = \langle F_\ast^e \operatorname{Perf}_{f^{-1}(Z)}(V) \rangle$. 
    Let $G\in \operatorname{Perf}_{f^{-1}(Z)}(V)$ be a classical generator. 
    Since $V\times_\mathcal{S} U$ is a quasi-affine scheme, $(s^\prime)^\ast G$ is a classical generator for $\operatorname{Perf}_{(f\circ s^\prime)^{-1}(Z)} (V\times_\mathcal{S} U)$. 
    To see, note that \cite[Lemma 4.8(3)]{Hall/Rydh:2017} implies
    \begin{displaymath}
        \operatorname{supp}(\mathbf{L} (s^\prime)^\ast G)= (s^\prime)^{-1}(\operatorname{supp}(G))=(f\circ s^\prime)^{-1}(Z).
    \end{displaymath}
    Then, for each $t\in (f\circ s^\prime)^{-1}(Z)$, \cite[Lemma 1.2]{Neeman:1992} ensures that
    \begin{displaymath}
        \langle (\mathbf{L} (s^\prime)^\ast G)_t \rangle = \operatorname{Perf}_{(f\circ s^\prime)^{-1}(Z) \cap \operatorname{Spec}(\mathcal{O}_{V\times_\mathcal{S} U,t})} (\operatorname{Spec}(\mathcal{O}_{V\times_\mathcal{S} U,t})).
    \end{displaymath}
    However, applying \Cref{lem:generation_is_affine_local} and \cite[Theorem 3.6]{Letz:2021}, it follows that $\mathbf{L} (s^\prime)^\ast G$ is a classical generator for $\operatorname{Perf}_{(f\circ s^\prime)^{-1}(Z)} (V\times_\mathcal{S} U)$ as $\mathcal{O}_{V\times_\mathcal{S} U}$ is a compact generator of $D_{\operatorname{qc}}(V\times_\mathcal{S} U,t)$ (using that $V\times_\mathcal{S} U$ is a quasi-affine scheme).

    Next, \Cref{prop:relative_Frobenius_generation_schemes} implies that for each $E\in D^b_{\operatorname{coh},f^{-1}(Z)}(V)$, 
    \begin{displaymath}
        ((\mathbf{L} s^\prime)^\ast E)_t \in \langle (\mathbf{L} (s^\prime)^\ast F_\ast^e G)_t \rangle = \langle (F_\ast^e \mathbf{L} (s^\prime)^\ast G)_t \rangle = \langle F_\ast^e (\mathbf{L} (s^\prime)^\ast G)_t \rangle
    \end{displaymath}
    where the second equality uses \Cref{rm:etale_base_change_Frobenii} and the third equality uses that Frobenius commutes with localization.
    There is a commutative square
    \begin{displaymath}
        \begin{tikzcd}
            {\operatorname{Spec}(\mathcal{O}_{V\times_\mathcal{S} U,t})} & {\operatorname{Spec}(\mathcal{O}_{V,s^\prime(t)})} \\
            {V\times_\mathcal{S} U} & V
            \arrow[from=1-1, to=1-2]
            \arrow[from=1-1, to=2-1]
            \arrow[from=1-2, to=2-2]
            \arrow["{s^\prime}"', from=2-1, to=2-2]
        \end{tikzcd}
    \end{displaymath} 
    of natural morphisms. 
    From the work above, we see the derived pullback of $E_{f(t)}$ along the faithfully flat morphism of affine schemes $\operatorname{Spec}(\mathcal{O}_{V\times_\mathcal{S} U,t}) \to \operatorname{Spec}(\mathcal{O}_{V,s^\prime(t)})$ belongs to $\langle F_\ast^e (\mathbf{L} (s^\prime)^\ast G)_t \rangle$. 
    According to \cite[Corollary 2.16]{Letz:2021}, it follows that $E_{s^\prime(t)} \in \langle (F_\ast^e G)_t \rangle$. 
    However, $s^\prime$ is surjective, so this holds for all points of the affine scheme $V$. 
    Consequently, \cite[Theorem 3.6]{Letz:2021} implies $E \in \langle F_\ast^e G \rangle$, which completes the proof.
\end{proof}

\begin{remark}
    \label{rmk:stacky_curve}
    Recall that a \textbf{stacky curve} is a smooth, proper, geometrically connected Deligne--Mumford
    stack $\mathcal{X}$ of Krull dimension one over a field $k$ that is generically a scheme \cite[Definition 5.2.1]{Voight/Zureick-Brown:2022}. The latter condition means there is an open dense substack which is a scheme. 
    Now, the coarse moduli space $\pi\colon \mathcal{X} \to X$ of the stacky curve is a smooth projective curve over $k$. 
    Indeed, \cite[Lemma 5.3.4]{Voight/Zureick-Brown:2022} says $X$ is smooth over $k$, whereas \cite[Theorem 6.12]{Rydh:2013} and \cite[\href{https://stacks.math.columbia.edu/tag/0A26}{Tag 0A26}]{stacks-project} ensure $X$ is projective over $k$. 
\end{remark}

\begin{lemma}
    \label{lem:good_moduli_dense}
    Let $\pi \colon \mathcal{X} \to X$ be a good moduli space (see \cite[\S 4]{Alper:2013}) with Noetherian source and target. 
    Then $\mathbf{R}\pi_\ast\colon D^b_{\operatorname{coh}}(\mathcal{X}) \to D^b_{\operatorname{coh}}(X)$ is essentially surjective.
\end{lemma}

\begin{proof}
    The following idea appeared in the proof of \cite[Lemma 2.17]{Ballard/Favero:2012}, but we spell it out in slightly more generality. 
    By \cite[Proposition 4.5]{Alper:2013}, the unit of the underived pullback/pushforward adjunction for $\pi$ is an isomorphism on quasi-coherent sheaves. 
    Moreover, the pushforward functor of $\pi$ is exact on quasi-coherent sheaves. 
    Hence, we may identify $\mathbf{R}\pi_\ast$ with the assignment of applying $\pi_\ast$ to each component of a complex. 
    Therefore we may apply the unit of the underived pullback/pushforward adjunction for $\pi$ termwise to obtain a quasi-isomorphism for every complex with quasi-coherent cohomology sheaves. 
    Thus, the desired claim follows.
\end{proof}

\begin{proposition}
    \label{prop:F-thickness_stacky_curve}
    Consider a tame stacky curve $\mathcal{X}$ over an algebraically closed field of prime characteristic. 
    Denote by $\pi\colon \mathcal{X} \to X$ the associated coarse moduli space. Suppose there is an $e>0$ such that $\mathbf{R} F_\ast^e \mathcal{O}_{\mathcal{X}}$ is a strong generator for $D^b_{\operatorname{coh}}(\mathcal{X})$. 
    Then the curve $X$ has genus zero. 
\end{proposition}

\begin{proof}
    By \cite[Theorem 4.10]{Ballard/Iyengar/Lank/Mukhopadhyay/Pollitz:2023}, it suffices to show that $F_\ast^e \mathcal{O}_X$ is a strong generator for $D^b_{\operatorname{coh}}(X)$. 
    As $\mathcal{X}$ is tame, \cite[Example 8.1]{Alper:2013} tells us $\pi$ is a good moduli space.
    There is a string of isomorphisms,
    \begin{displaymath}
        \mathbf{R}\pi_\ast \mathbf{R} F_\ast^e \mathcal{O}_{\mathcal{X}} \cong F_\ast^e \mathbf{R} \pi_\ast \mathcal{O}_{\mathcal{X}} \cong F_\ast^e \mathcal{O}_X
    \end{displaymath}
    where we have used the isomorphism $\mathcal{O}_X \to \mathbf{R} \pi_\ast \mathcal{O}_{\mathcal{X}}$ coming from the unit of derived pullback/pushforward adjunction. 
    If we can show $\pi_\ast \mathbf{R} F_\ast^e \mathcal{O}_{\mathcal{X}}$ is a strong generator for $D^b_{\operatorname{coh}}(X)$, then we are done. However, \Cref{lem:good_moduli_dense} tells us $\mathbf{R} \pi_\ast$ is essentially dense on $D^b_{\operatorname{coh}}$ and the result follows.
\end{proof}

\appendix

\section{Verdier localizations with support}

We record a Verdier localization sequence for the bounded derived category of coherent sheaves with support on Noetherian algebraic stacks. 
In the case without support, this is \cite[Proposition B.1]{Hall/Lamarche/Lank/Peng:2025}. 
Our argument for the support case follows closely from that of \cite[Proposition B.1]{Hall/Lamarche/Lank/Peng:2025}. 

\begin{notation}
    Let $\mathcal{X}$ be a Noetherian algebraic stack. 
    We write $\overline{\operatorname{coh}}(\mathcal{X})\subseteq\operatorname{Qcoh}(\mathcal{X})$ for the full subcategory consisting of countably generated quasi-coherent $\mathcal{O_X}$-modules. 
    Set $D_{\overline{\operatorname{coh}}}^b(\mathcal{X})\subseteq D_{\operatorname{qc}}(\mathcal{X})$ to be the full subcategory consisting of bounded complexes with countably generated cohomology sheaves. 
\end{notation}
 
\begin{remark}
    By \cite[Lemma A.1]{Hall:2022}, $\overline{\operatorname{coh}}(\mathcal{X})\subseteq\operatorname{Qcoh}(\mathcal{X})$ is a Serre subcategory. 
    It follows that $D_{\overline{\operatorname{coh}}}^b(\mathcal{X})$ is a triangulated subcategory of $D_{\operatorname{qc}}(\mathcal{X})$.
\end{remark}

\begin{lemma}
    \label{lem:relative_denseness_open_immersion}
    Let $\mathcal{X}$ be a Noetherian algebraic stack. 
    Consider an open immersion $j\colon \mathcal{U} \to \mathcal{X}$. 
    Suppose $Z\subseteq |\mathcal{U}|$ is closed. 
    Denote by $\overline{Z}$ for the closure of $Z$ in $|\mathcal{X}|$. 
    Then $\mathbf{L} j^\ast$ restricts to a Verdier localization $D^b_{\operatorname{coh},\overline{Z}} (\mathcal{X}) \to D^b_{\operatorname{coh},Z}(\mathcal{U})$ whose kernel consists of $E\in D^b_{\operatorname{coh},\overline{Z}} (\mathcal{X})$ with support is contained in $Z\cap (|\mathcal{X}|\setminus |\mathcal{U}|)$.
\end{lemma}

\begin{proof}
    By \cite[Proposition B.1]{Hall/Lamarche/Lank/Peng:2025}, we have the following localization sequence
    \begin{displaymath}
        D^b_{\operatorname{coh},|\mathcal{X}|\setminus|\mathcal{U}|}(\mathcal{X})
        \longrightarrow D^b_{\operatorname{coh}} (\mathcal{X}) \to D^b_{\operatorname{coh}}(\mathcal{U}).
    \end{displaymath}
    Let $M\in D^b_{\operatorname{coh},Z} (\mathcal{U})$. 
    We first claim that there exists $\overline{M}\in D^b_{\operatorname{coh},\overline{Z}} (\mathcal{X})$ such that $\mathbf{L} j^\ast\overline{M}\simeq M$. 
    Indeed, we may write $\mathbf{R} j_\ast M\simeq \operatorname{hocolim}_s M_s$ for some $M_s\in D^b_{\operatorname{coh}} (\mathcal{X})$ and the induced morphisms $\mathcal{H}^{i}(M_s)\to\mathcal{H}^i(\mathbf{R} j_\ast M)$ are monomorphisms for every $i\in\mathbb{Z}$ \cite[Lemma A.3]{Hall:2022}. Since $M$ is supported on $Z$, we see that $\mathbf{R} j_\ast M$ is supported on $\overline{Z}$. 
    It follows that $M_s$ is supported on $\overline{Z}$ for all $s$. 
    Note that $j$ is an open immersion and thus flat. 
    The composition $j^\ast M_s\to j^\ast\mathbf{R} j_\ast M\simeq M$ induces a monomorphism $\mathcal{H}^{i}(M_s)\to\mathcal{H}^i(\mathbf{R} j_\ast M)$ for every $i$. 
    Since $M\in D^b_{\operatorname{coh}}(\mathcal{U})$, we have $j^\ast M_s\simeq M$ for $s\gg 0$. 
    As $M_s$ is supported on $\overline{Z}$, the claim follows and $D^b_{\operatorname{coh},\overline{Z}} (\mathcal{X}) \to D^b_{\operatorname{coh},Z}(\mathcal{U})$ is essentially surjective.

    It remains to show that the kernel of $D^b_{\operatorname{coh},\overline{Z}} (\mathcal{X}) \to D^b_{\operatorname{coh},Z}(\mathcal{U})$ is $D^b_{\operatorname{coh},Z\cap (|\mathcal{X}|\setminus |\mathcal{U}|)} (\mathcal{X})$. 
    To this end, let $\phi\colon M\to N$ be a morphism in $D^b_{\overline{\operatorname{coh}}} (\mathcal{X})\cap D_{\operatorname{qc},\overline{Z}}(\mathcal{X})$ such that $N\in D^b_{\operatorname{coh},\overline{Z}} (\mathcal{X})$ and $j^\ast\phi$ is an isomorphism. 
    We claim that there exists $\psi\colon M^\prime\to M$ such that $j^\ast\psi$ is also an isomorphism and $M^\prime\in D^b_{\operatorname{coh},\overline{Z}} (\mathcal{X})$. 
    To see this claim, we form the following triangle
    \begin{displaymath}
        K\longrightarrow M\overset{\phi}{\longrightarrow} N\longrightarrow K[1].
    \end{displaymath}
    By assumption, we have $j^\ast K\simeq 0$ and thus $\operatorname{supp}(K)\subseteq\overline{Z}\cap(|\mathcal{X}|\setminus |\mathcal{U}|)=Z\cap (|\mathcal{X}|\setminus |\mathcal{U}|)$. 
    Applying \cite[Lemma A.3]{Hall:2022}, we may write $K\simeq\operatorname{hocolim}_s K_s$ such that $K_s\in D^b_{\operatorname{coh}} (\mathcal{X})$ and the induced morphism $\mathcal{H}^{i}(K_s)\to\mathcal{H}^i(K)$ is a monomorphism for every $i$. 
    It follows that $K_s\in D^b_{\operatorname{coh},Z\cap (|\mathcal{X}|\setminus |\mathcal{U}|)} (\mathcal{X})$ for every $s$. Since $N\in D^b_{\operatorname{coh},\overline{Z}} (\mathcal{X})$, the morphism $N\to K[1]$ factors through $K_s[1]\to K[1]$ for some $s\gg 0$ by \cite[Lemma 1.2]{Hall/Rydh:2017}. 
    Form the following triangle
    \begin{displaymath}
        K_s\longrightarrow M^\prime\longrightarrow N\longrightarrow K[1].
    \end{displaymath}
    By construction, we have $M^\prime\in D^b_{\operatorname{coh},\overline{Z}} (\mathcal{X})$.
    Moreover, we obtain a morphism $\psi\colon M^\prime\to M$ by TR3 and the claim follows. 
    This finishes the proof.
\end{proof}

\bibliographystyle{alpha}
\bibliography{mainbib}

\end{document}